\documentclass[10pt,a4paper]{article}

\usepackage{a4wide}

\usepackage{amsmath}
\usepackage{amssymb}
\usepackage{amsthm}
\usepackage[normalem]{ulem}
\usepackage{color}
\usepackage[left=3.0cm,top=3.5cm,bottom=4.0cm]{geometry}
\usepackage{graphicx}

\usepackage{appendix}

\numberwithin{equation}{section}

\begin{document}

\newtheorem{theorem}{Theorem}[section]
\newtheorem{lemma}[theorem]{Lemma}
\newtheorem{definition}[theorem]{Definition}
\newtheorem{proposition}[theorem]{Proposition}
\newtheorem{corollary}[theorem]{Corollary}
\newtheorem{example}[theorem]{Example}
\newtheorem{remark}[theorem]{Remark}
\newtheorem{conjecture}[theorem]{Conjecture}
\newtheorem*{assumption}{Theorem}

\title{Generalised Particle Filters with Gaussian Mixtures}
\author{Dan Crisan\thanks{Department of Mathematics, Imperial College London, London, SW7 2AZ, UK. Email: d.crisan@imperial.ac.uk},\and Kai Li\thanks{Department of Mathematics, Uppsala University, Box 480, Uppsala, 75106, Sweden. Email: kai.li@math.uu.se}}

\maketitle

\begin{abstract}
Stochastic filtering is defined as the estimation of a partially observed dynamical system. A massive scientific and computational effort
is dedicated to the development of numerical methods for approximating the
solution of the filtering problem. Approximating the solution of the filtering problem with Gaussian mixtures has been a very popular method since the 1970s (see \cite{Alspach and Sorenson},\cite{Anderson and Moore},\cite{Sorenson and Alspach},\cite{Tam and Moore}). Despite nearly fifty years of development, the existing work is based on the success of the numerical implementation and is not theoretically justified. This paper fills this gap and contains a rigorous analysis of a new Gaussian mixture approximation to the solution of the filtering problem. We deduce the $L^2$-convergence rate for the approximating system and show some numerical example to test the new algorithm. 
\end{abstract}

\section{Introduction}

The stochastic filtering problem deals with the estimation of an evolving dynamical system, called the \emph{signal}, based on \emph{partial observations} and a priori stochastic model. The signal is modelled by a stochastic process denoted by $X=\{X_t,\ t\geq0\}$, defined on a probability space $(\Omega,\mathcal{F},\mathbb{P})$. The signal process is not available to observe directly; instead, a partial observation is obtained and it is modelled by a process $Y=\{Y_t,\ t\geq0\}$.
The information available from the observation up to time $t$ is defined as the filtration $\mathcal{Y}=\{\mathcal{Y}_t,\ t\geq0\}$ generated by the observation process $Y$.
In this setting, we want to compute $\pi_t$ --- the conditional distribution of $X_t$ given $\mathcal{Y}_t$. 
The analytical solution to the filtering problem are rarely available, and there are only few exceptions such as the Kalman-Bucy filter and the Bene\v s filter (see, e.g. Chapter 6 in \cite{Bain and Crisan}). Therefore numerical algorithms for solving the filtering equations are required. 

The description of a numerical approximation for $\pi_t$ should contain the following three parts: the class of approximations; the law of evolution of the approximation; and the method of measuring the approximating error.  
Generalised particle filters with Gaussian mixtures is a numerical scheme to approximate the solution of the filtering problem, and it is a natural generalisation of the classic particle filters (or sequential Monte Carlo methods) in the sense that the Dirac delta measures used in the classic particle approximations are replaced by mixtures of Gaussian measures. In other words, Gaussian mixture approximations are algorithms that approximate $\pi_t$ with random measures of the form
$$
\sum_ia_i(t)\Gamma_{v_i(t),\omega_j(t)},
$$
where $\Gamma_{v_j(t),\omega_j(t)}$ is the Gaussian measure with mean $v_j(t)$ and covariance matrix $\omega_j(t)$. The evolution of the weights, the mean and the covariance matrices satisfy certain stochastic differential equations which are numerically solvable. 
As time increases, typically the trajectories of a large number of particles diverge from the signal's trajectory; with only a small number remaining close to the signal. The weights of the diverging particles decrease rapidly, therefore contributing very little to the approximating system, and causing the approximation to converge very slowly to the conditional distribution. In order to tackle this so-called \emph{sample degeneracy} phenomenon, a \emph{correction procedure} is added. At correction times, each particle is replaced by a random number of offspring. Redundant particles are abandoned and only the particles contributing significantly to the system (i.e. with large weights) are carried forward; so that the most probable region of the trajectory of the signal process $X$ will be more thoroughly explored. This correction mechanism is also called branching or resampling. Currently the multinomial branching algorithm and the tree based branching algorithm (TBBA) are two approaches for the correction step.

The idea of using Gaussian mixtures in the context of Bayesian estimation can be traced back to Alspach and Sorenson (\cite{Alspach and Sorenson}, \cite{Sorenson and Alspach}), and Tam and Moore (\cite{Tam and Moore}). Later in the work by Anderson and Moore (\cite{Anderson and Moore}), Gaussian mixtures are used in combination with extended Kalman filters to produce an empirical approximation to the discrete time filtering problem. The research on this area became more active since 2000s. Recently in Doucet et al (\cite{Doucet et al 2}), these methods are revisited to construct a method which uses Rao-Blackwellisation in order to take advantage of the analytic structure present in some important classes of state-space models. Recent advances in this direction are contained in \cite{Andrieu and Doucet}, \cite{Doucet et al 3} and \cite{Freitas}.
In Chen and Liu (\cite{Chen and Liu}), a random mixture of Gaussian distributions, called mixture Kalman filters, are used to approximate a target distribution again based on the explicit linear filter. Further work on this topic is contained in \cite{Guo et al},  \cite{Sun et al}, \cite{Morelande and Challa}, and \cite{Wu et al}. Gustafsson et al (\cite{Gustafsson et al}) describes a general framework for a number of applications, which are implemented using the idea of Gaussian particle filters. Further development in this direction can be found in  \cite{Djuric et al} ,\cite{Fox}, \cite{Gustafsson} and \cite{Kwok et al}. For more recent work related to Gaussian mixture approximations, see Kotecha and Djuri\'c (\cite{Kotecha and Djuric}, \cite{Kotecha and Djuric2}), Le Gland et al (\cite{Le Gland and Monbet and Tran}), Reich (\cite{Reich}), Flament et al (\cite{Flament and Fleury and Davoust}), Van der Merwe and Wan (\cite{Van der Merwe and Wan}), Carmi et al (\cite{Carmi and Septier and Godsill}), and Iglesias, Law and Stuart (\cite{Iglesias and Law and Stuart}).

\subsection{Contribution of the paper}
In this paper we construct a new approximation to the conditional distribution of the signal $\pi=\{\pi_t:t\geq0\}$ that consists of a mixture of Gaussian measures. In contrast with the existing work, the approximation is not based on the (Extended) Kalman filter.
The approximation is analysed theoretically: we obtain the rate of convergence of the approximation to $\pi_t$ as the number of Gaussian measures increases. We are not aware of other theoretically justified algorithms of this kind.

In particular, if $\pi^n=\{\pi_t^n:t\geq0\}$ is the constructed approximation of $\pi_t$, where $n$ is the number of Gaussian measures, we prove the following result.
\begin{assumption}
\textbf{For any $T\geq0$ and $m\geq6$, there exists a constant $C(T,\varphi)$ independent of $n$, such that for any test function $\varphi\in C_b^{m+2}(\mathbb R^d)$,
\begin{equation}
\mathbb E\left[\sup_{t\in[0,T]}\left|\pi_t^n(\varphi)-\pi_t(\varphi)\right|\right]\leq\frac{C(T,\varphi)}{\sqrt n}.
\end{equation}} 
\end{assumption}
Additional $L^p$ result can also be covered. The proof of this result is harder than the proof of the convergence of approximating measures under the classic particle filters. The technical difficulty arise from the fact that the covariance matrix of the component Gaussian measure is random and as a result we cannot use a standard approach such as that contained in the proof of the convergence of the classic particle filters by using the dual of the conditional distribution process (see, e.g. Chapter 7 in \cite{Bain and Crisan}). We deal with this difficulty by adopting the ideas in \cite{Obanubi}, namely we show a variation of Theorem 8 in \cite{Obanubi} (see Theorem \ref{prop.general_convergence_result}), which is an abstract convergence criterion for signed measure-valued process. We then verify the conditions of Theorem \ref{prop.general_convergence_result} hold in our case so that it can be applied to prove the convergence of the Gaussian mixture approximating measure. 

We also make use of the Bene\v s filter as an example to numerically compare the performances of the Gaussian mixture approximation and the classic particle approximation. The implementation shows that generally Gaussian mixtures approximation performs better than particle approximations when the number of particles/Gaussian mixtures is relatively small.\\

The following is a summary of the contents of the paper.

In Section 2, we review the central results of stochastic filtering theory\footnote{For a short historical account of the development of stochastic filtering problem, see, for example, Section 1.3 of \cite{Bain and Crisan}.}. The filtering framework is introduced first, with the focus on the problems where the signal $X$ and observation $Y$ are diffusion processes and the filtering equations are presented.

Section 3 contains the description of the \emph{generalised particle filters with Gaussian mixtures}. These approximations use mixtures of Gaussian measures which will be set out, with the aim of estimating the solutions to the Zakai and the Kushner-Stratonovich equations. The Tree Based Branching Algorithm (TBBA) and Multinomial branching algorithm are discussed as possible correction mechanisms..

Section 4 contains the main results of the thesis, which is the law of large numbers theorem associated to the approximating system. In this secrion, the evolution equations of the approximating systems introduced in the previous chapter are derived. It is shown that, under certain conditions, the unnormalised and normalised versions of the approximations of the conditional distribution converge to the solutions of the Zakai equation and the Kushner-Stratonovich equation, respectively. Section 5 contains a numerical example to compare the approximations by using the classic particle filters and the Gaussian mixtures.

This paper is concluded with an Appendix which contains some additional proofs.

\subsection{Notations}
\begin{itemize}
\item
$(\Omega,\mathcal F,\mathbb P)$ - probability triple consisting of a sample
space $\Omega$, the $\sigma$-algebra $\mathcal F$ which
is the set of all measurable events, an the probability measure $\mathbb
P$.
\item
$(\mathcal F_t)_{t\geq0}$ - a filtration, an increasing family of sub-$\sigma$-algebras
of $\mathcal F$; $\mathcal F_s\subset\mathcal F_t$, $0\leq s\leq t$.
\item
$\mathbb R^d$ - the $d$-dimensional Euclidean space.
\item
$\overline{\mathbb R^d}$ - the one-point compactification of $\mathbb R^d$
formed by adding a single point at infinity to $\mathbb R^d$.
\item
$\left(\mathbb S,\mathcal B(\mathbb S)\right)$ - the state space of the signal.
Normally $\mathbb S$ is taken as a complete separable space, and $\mathcal
B(\mathbb S)$ is the associated Borel $\sigma$-algebra, that is, the $\sigma$-algebra
generated by the open sets in $\mathbb S$.
\item
$B(\mathbb S)$ - the space of bounded $\mathcal B(\mathbb S)$-measurable
functions from $\mathbb S$ to $\mathbb R$.
\item
$\mathcal P\left(\mathbb S\right)$ - the family of Borel probability measures
on space $\mathbb S$.
\item
$C_b(\mathbb R^d)$ - the space of bounded continuous functions on $\mathbb
R^d$.
\item
$C_b^m(\mathbb R^d)$ - the space of bounded continuous functions on $\mathbb
R^d$ with bounded derivatives up to order $m$.
\item
$C_0^m(\mathbb R^d)$ - the space of continuous functions on $\mathbb R^d$,
vanishing at infinity with continuous partial derivatives up to order $m$.
\item
$\Vert\cdot\Vert$ - the Euclidean norm for a $d\times p$ matrix $a$, $\Vert a\Vert=\sqrt{\sum_{i=1}^d\sum_{j=1}^pa_{ij}^2}$.
\item
$\|\cdot\|_\infty$ - the supremum norm for $\varphi:\mathbb R^d\rightarrow\mathbb
R$: $\|\varphi\|_\infty=\sup_{x\in\mathbb R^d}\|\varphi(x)\|$.
\item
$\|\cdot\|_{m,\infty}$ - the norm such that for  $\varphi$ on $\mathbb R^d$,
$\|\varphi\|_{m,\infty}=\sum_{|\alpha|\leq m}\sup_{x\in\mathbb R^d}|D_\alpha\varphi(x)|$,
where $\alpha=(\alpha^1,\ldots,\alpha^d)$ is a multi-index and $D_\alpha=(\partial_1)^{\alpha_1}\cdots(\partial_d)^{\alpha_d}$.
\item
$\mathcal M_{F}(\mathbb R^d)$ - the set of finite measures on $\mathbb R^d$.
\item
$\mathcal M_{F}(\overline{\mathbb R^d})$ - the set of finite measures on
$\overline{\mathbb R^d}$.
\item
$D_{\mathcal M_{F}(\mathbb R^d)}[0,T]$ - the space of c\`adl\`ag functions
(or right continuous functions with left limits) $f:[0,T]\rightarrow\mathcal
M_{F}(\mathbb R^d)$.
\item
$D_{\mathcal M_{F}(\mathbb R^d)}[0,\infty)$ - the space of c\`adl\`ag functions
(or right continuous functions with left limits) $f:[0,\infty)\rightarrow\mathcal
M_{F}(\mathbb R^d)$.
\end{itemize}

\section{The Filtering Problem and Key Result}
Let $(\Omega,\mathcal{F},\mathbb{P})$ be a probability space together with a filtration $(\mathcal{F}_t)_{t\geq0}$ which satisfies the usual conditions. On $(\Omega,\mathcal{F},\mathbb{P})$ we consider a $\mathcal{F}_t$-adapted process $X=\{X_t; t\geq0\}$  taking value on $\mathbb{R}^d$.
To be specific,
let $X=(X^i)_{i=1}^d$ be the solution of a $d$-dimensional stochastic differential
equation driven by a $p$-dimensional Brownian motion $V=(V^j)_{j=1}^p$:
\begin{equation}\label{eq.signal_is_diffusion}
X_t^i=X_0^i+\int_0^tf^i(X_s)ds+\sum_{j=1}^p\int_0^t\sigma^{ij}(X_s)dV_s^j,\quad\quad
i=1,\ldots,d
\end{equation}
We assume that both $f=(f^i)_{i=1}^d:\mathbb{R}^d\rightarrow\mathbb{R}^d$
and $\sigma=(\sigma^{ij})_{i=1,\ldots,d;j=1,\ldots,p}:\mathbb{R}^d\rightarrow\mathbb{R}^{d\times
p}$ are globally Lipschitz. Under the globally Lipschitz condition, \eqref{eq.signal_is_diffusion}
has a unique solution (Theorem 5.2.9 in \cite{Karatzas and Shreve}). The
infinitesimal generator $A$ associated with the process $X$ is the second-order differential
operator
\begin{equation}\label{eq.generator_A}
A=\sum_{i=1}^d f^i\frac{\partial}{\partial x^i}+\sum_{i=1}^d\sum_{j=1}^da^{ij}\frac{\partial^2}{\partial
x^i\partial x^j},
\end{equation}   
where $a=(a^{ij})_{i,j=1,\ldots,d}:\mathbb{R}^d\rightarrow\mathbb{R}^{d\times
d}$ is the matrix-valued function defined as $a=\frac{1}{2}\sigma^\top\sigma$. 
We denote by $\mathcal{D}(A)$ the domain of $A$. 

Let $h=(h_i)_{i=1}^m:\mathbb{R}^d\rightarrow\mathbb{R}^m$ be a bounded measurable function. Let $W$ be a standard $\mathcal{F}_t$-adapted $m$-dimensional Brownian motion on  $(\Omega,\mathcal{F},\mathbb{P})$ independent of $X$, and $Y$ be the process which satisfies the following evolution equation
\begin{equation}\label{eq.observation}
Y_t=Y_0+\int_0^th(X_s)ds+W_t,
\end{equation}
This process $Y=\{Y_t;t\geq0\}$ is called the \textit{observation} process. Let $\{\mathcal{Y}_t,t\geq0\}$ be the usual augmentation of the filtration associated with the process $Y$, viz
$\mathcal{Y}_t=\sigma(Y_s,s\in[0,t])\vee\mathcal{N}.$

As stated in the introduction, the filtering problem consists in determining the conditional distribution $\pi_t$ of the signal $X$ at time $t$ given the information accumulated from observing $Y$ in the interval $[0,t]$; that is, for $\varphi\in B(\mathbb{R}^d)$,
\begin{equation}
\pi_t(\varphi)=\int_{\mathbb{S}}\varphi(x)\pi_t(dx)=\mathbb{E}[\varphi(X_t)\mid\mathcal{Y}_t].
\end{equation}

Throughout this paper we will assume that the coefficients $f^i$, $\sigma^{ij}$ and $h^i$ are all bounded and Lipschitz; $f^i$ and $\sigma^{ij}$ are six times differentiable, and $h^i$ is twice differentiable. In other words, we assume that $f^i, \sigma^{ij}\in C_b^6(\mathbb R^d)$ and $h^i\in C_b^2(\mathbb R^d)$. 

Let $\tilde{\mathbb P}$ be a new probability measure on $\Omega$, under which the process $Y$ is a Brownian motion. To be specific, let $Z=\{Z_t,t\geq0\}$ be the process defined by
\begin{equation}\label{eq.definition_Z}
Z_t=\exp\left(-\sum_{i=1}^m\int_0^t h^i(X_s)dW_s^i-\frac{1}{2}\sum_{i=1}^m\int_0^t h^i(X_s)^2ds\right),\quad t\geq0;
\end{equation}
and we introduce a probability measure $\tilde{\mathbb{P}}^t$ on $\mathcal{F}_t$ by specifying its Radon-Nikodym derivative with respect to $\mathbb{P}$ to be given by $Z_t$.
We finally define a probability measure $\tilde{\mathbb{P}}$ which is equivalent to $\mathbb{P}$ on $\bigcup_{0\leq t<\infty}\mathcal{F}_t$.
Then we have the following Kallianpur-Striebel formula (see \cite{Kallianpur and Karandikar})
\begin{equation}
\pi_t(\varphi)=\frac{\rho_t(\varphi)}{\rho_t(\mathbf 1)}\quad\quad \tilde{\mathbb P}(\mathbb{P})-a.s.\quad\text{for}\ \varphi\in B(\mathbb{R}^d);
\end{equation}
where $\rho_t$ is an $\mathcal Y_t$-adapted measure-valued process satisfying the following \textbf{Zakai Equation} (see \cite{Zakai}).
\begin{equation}\label{eq.zakai_equation}
\rho_t(\varphi)=\pi_0(\varphi)+\int_0^t\rho_s(A\varphi)ds+\int_0^t\rho_s(\varphi h^\top)dY_s,\quad\tilde{\mathbb{P}}-a.s.\quad\forall t\geq0
\end{equation}
for any $\varphi\in\mathcal{D}(A)$. Also the process $\rho=\{\rho_t;t\geq0\}$ is called the unnormalised conditional distribution of the signal. The equation satisfied by $\pi$ is called the Kushner-Stratonovich equation.

In the following we will analyse the generalised particle filters with Gaussian mixtures, and we will give detailed analysis in the following sections. We denote by $\pi^n=\{\pi_t^n;t\geq0\}$ the approximating measure of the solution of the filtering problem, where $n$ is the number of Gaussian measures in the approximating system. Then the main convergence of $\pi^n$ is stated as follows:

\begin{theorem}\label{thm.main_result}
For any $T>0$ and $m\geq6$, there exists a constant $C(T)$ independent of $n$, such that for any test function $\varphi\in C_b^{m+2}(\mathbb R^d)$
\begin{equation}
\tilde{\mathbb E}\left[\sup_{t\in[0,T]}|\pi_t^n(\varphi)-\pi_t(\varphi)|\right]\leq\frac{C(T)}{\sqrt n}\|\varphi\|_{m+2,\infty}.
\end{equation}
\end{theorem} 

\section{The Approximating System with Gaussian Mixtures}\label{gaussian}

For ease of notations, we assume, hereinafter from this section, that the state space of the signal is one-dimensional. The approximating algorithm discussed in this section, together with the $L^2$-convergence analysis in Section 4 are done under this assumption. We note that all the results hereinafter can be extended without significant technical difficulties to the multi-dimensional case.

Firstly, we let $\Delta=\{0=\delta_0<\delta_1<\cdots<\delta_N=T\}$ be an equidistant partition of the interval $[0,T]$ with equal length,
with $\delta_i=i\delta,\ i=1,\ldots,N$; and $N=\frac{T}{\delta}$. The approximating algorithm is then introduced as follows.

\textbf{Initialisation}: At time zero, the particle system consists of $n$ Gaussian measures  all with equal weights $1/n$, initial means $v_j^n(0)$, and initial variances $\omega_j^n(0)$, for $j=1,\ldots,n$; denoted by 
$\Gamma_{v_j^n(0),\omega_j^n(0)}$. 
The approximation of $\pi_0$ has the form
\begin{equation}
\pi_0^n\triangleq{1\over n}\sum_{j=1}^n\Gamma_{v_j^n(0),\omega_j^n(0)},
\end{equation}

\textbf{Recursion}: During the interval $t\in[i\delta,(i+1)\delta)$, $i=1,\ldots,N,$ the approximation $\rho^n$ of the unnormalised conditional distribution $\rho$ will take the form
\begin{equation}\label{eq.gaussian_mixture_approximation}
\pi_t^n\triangleq\sum_{j=1}^n\bar a_j^n(t)\Gamma_{v_j^n(t),\omega_j^n(t)},
\end{equation}
where $v_j^n(t)$ denotes the mean and $\omega_j^n(t)$ denotes the variance of the Gaussian measure $\Gamma_{v_j^n(t),\omega_j^n(t)}$, and $a_j^n(t)$ is the (unnormalised) weight of the particle, and $$\bar a_j^n(t)=\frac{a_j^n(t)}{\sum_{k=1}^na_k^n(t)}$$
is the normalised weight. Obviously, each particle is characterised by the triple process $(a_j^n,v_j^n,\omega_j^n)$ which is chosen to evolve as
\begin{equation}\label{eq.appro_with_variance}
\left\{
\begin{array}{lll}
a_j^n(t)=1+\int_{i\delta}^t a_j^n(s)h(v_j^n(s))dY_s,\\
v_j^n(t)=v_j^n(i\delta)+\int_{i\delta}^tf\left(v_j^n(s)\right)ds
+\sqrt{1-\alpha}\int_{i\delta}^t\sigma\left(v_j^n(s)\right)dV_s^{(j)},\\
 \omega_j^n(t)=\alpha\left(\beta+\int_{i\delta}^t\sigma^2\left(v_j^n(s)\right)ds\right),
\end{array}
\right.
\end{equation}
where $\{{V^{(j)}}\}_{j=1}^n$ are mutually independent Brownian motions and  independent of $Y$. The parameter $\alpha$ is a real number in the interval $[0,1]$. For $\alpha=0$ we recover the classic particle approximation (see, for example, Chapter 9 in \cite{Bain and Crisan}); for $\alpha=1$ the mean of the Gaussian measures evolve deterministically (the stochastic term is eliminated). The parameter $\beta$ is a positive real number, which we call the \emph{smoothing parameter}, ensures that the approximating measure has smooth density at the branching time.

\textbf{Correction}: at the end of the interval $[i\delta,(i+1)\delta)$, immediately prior to the correction step, each Gaussian measure is replaced by a random number of offsprings, which are Gaussian measures with mean $X_j^n((i+1)\delta)$ and variance $\alpha\beta$, where the mean $X_j^n$ is a normally distributed random variable, i.e.
$$
X_j^n((i+1)\delta)\sim\mathcal N\left(v_j^n(i+1)\delta_-,\omega_j^n(i+1)\delta_-\right),\quad j=1,\ldots,n.
$$ 
We denote by $o_j^{n,(i+1)\delta}$ the number of ``offsprings'' produced by $j$th generalised particle. The total number of offsprings is fixed to be $n$ at each correcting event.

After correction all the particles are re-indexed from 1 to $n$ and all of
the unnormalised weights are re-initialised back to 1; and the particles
evolve following \eqref{eq.appro_with_variance} again. 
The recursion is repeated $N$ times until we reach the terminal time $T$,
where we obtain the approximation $\pi_T^n$ of $\pi_T$.

Before discussing how the correction step is actually carried out, we give here a brief explanation why we should introduce it. As time increases, the unnormalised weights of the majority of the particles decrease to zero, with only few becoming very large (or equivalently, the normalised weights of the majority of the particles decrease to zero, with only few becoming close to one), this phenomenon is called the \emph{sample degeneracy}.  As a consequence, only a small number of particles contribute significantly to the approximations, and therefore a large number of particles are needed in order to obtain the required accuracy; in other words, the convergence of this approximation is very slow. In order to solve this, a correction procedure is used which culls particles with small weights and multiplies particles with large weights.
The resampling depends both on the weights of the particles and the observation data, and by doing this particles with small weights (and hence their trajectories are far from the signal) are not carried forward and therefore the more likely region where the signal might be can be explored.

In the following we discuss two correction mechanisms. The first one uses the so called Tree Based Branching Algorithm (TBBA) and the second one is based on the Multinomial Resampling to determine the number of offsprings $\{o_j^n\}_{j=1}^n$.\\

The Tree Based Branching Algorithm (see Chapter 9 in \cite{Bain and Crisan}) produces offsprings $\{o_j^n\}$ with distribution
\begin{equation}\label{eq.number_of_offspring}
o_{j}^{n,(i+1)\delta}=\left\{
\begin{array}{lll}
\left[n\bar a_{j}^{n,(i+1)\delta}\right]\qquad\text{with prob.}\quad1-\{n\bar a_{j}^{n,(i+1)\delta}\}\\
\\
\left[n\bar a_{j}^{n,(i+1)\delta}\right]+1\qquad\text{with prob.}\quad\{n\bar a_{j}^{n,(i+1)\delta}\};
\end{array}
\right.
\end{equation}
where $\bar a_j^{n,(i+1)\delta}$ is the value of the Gaussian particle's weight immediately prior to the branching, in other words,
$$
\bar a_j^{n,(i+1)\delta}=\bar a_j^n((i+1)\delta-)=\lim_{t\nearrow(i+1)\delta}\bar a_j^n(t).
$$
If $\mathcal F_{(i+1)\delta-}$ is the $\sigma$-algebra of events up to time $(i+1)\delta$, i.e.
$\mathcal F_{(i+1)\delta-}=\sigma(\mathcal F_s:s<(i+1)\delta),$
then we have the following proposition (see Chapter 9 in \cite{Bain and Crisan}).
\begin{proposition}\label{prop.tbba_proposition} 
The random variables $\{o_j^n\}_{j=1}^n$ defined in \eqref{eq.number_of_offspring} have the following properties
\begin{align}\label{eq.expectation_of_offspring}
&\mathbb{E}\left[o_{j'}^{n,(i+1)\delta}\vert\mathcal{F}_{(i+1)\delta-}\right]=n\bar a_{j'}^{n,(i+1)\delta},\nonumber\\
&\mathbb{E}\left[\left(o_{j}^{n,(i+1)\delta}-n\bar a_j^{n,(i+1)\delta}\right)^2\big|\mathcal F_{(i+1)\delta-}\right]=
\left\{n\bar a_{j}^{n,(i+1)\delta}\right\}\left(1-\left\{n\bar a_{j}^{n,(i+1)\delta}\right\}\right).
\nonumber
\end{align}
\end{proposition}
\begin{remark}
The random variables $o_j^{n,(i+1)\delta}$ defined \eqref{eq.number_of_offspring} have conditional minimal variance in the set of all integer-valued random variables $\xi$ satisfying $\mathbb E[\xi|\mathcal F_{(i+1)\delta-}]=n\bar a_j^{n,(i+1)\delta}$. This property is important as it is the variance of $o_j^n$ that influences the speed of the corresponding algorithm (see Exercise 9.1 in \cite{Bain and Crisan}).
\end{remark}
In addition the TBBA keeps the number of particles in the system constant at $n$; that is, for each $i$,
\begin{equation}\label{eq.fix_total_offspring_number}
\sum_{j=1}^no_j^{n,(i+1)\delta}=n.
\end{equation}
The TBBA is, for example, discussed in Section 9.2.1 in \cite{Bain and Crisan} to ensure \eqref{eq.fix_total_offspring_number} is satisfied, and by Proposition 9.3 in \cite{Bain and Crisan} we know that the distribution of $o_j^{n}$ satisfies \eqref{eq.number_of_offspring} and Proposition \ref{prop.tbba_proposition}.\\

If multinomial resampling is used (see, for example, \cite{Obanubi}), 
then the offspring distribution is determined by the multinomial distribution
$$O_{(i+1)\delta}=\text{Multinomial}(n,\bar a_1^n((i+1)\delta_-),\ldots,\bar a_n^n((i+1)\delta_-)),$$
i.e.
\begin{equation}
\mathbb P\left(O_{(i+1)\delta}^{(j)}=o_j^{n,(i+1)\delta},j=1,\ldots,n\right)=
\frac{n!}{\prod_{j=1}^no_j^{n,(i+1)\delta}!}\prod_{j=1}^n\left(\bar a_j^n((i+i)\delta-)\right)^{o_j^{n,(i+1)\delta}}
\end{equation}
with $\sum_{j=1}^no_j^{n,(i+1)\delta}=n$.

By properties of the multinomial distribution, we have the following result (see, for example, \cite{Li}).
\begin{proposition}\label{prop.multinomial_proposition}
At branching time $(i+1)\delta$, $\left\{O_{(i+1)\delta}^{(j)}=o_j^{n,(i+1)\delta}\right\}_{j=1}^n$ has a multinomial distribution, then the conditional mean is proportional to the normalised weights of their parents:
\begin{align}
\tilde{\mathbb E}\left[o_j^{n,(i+1)\delta}\big|\mathcal F_{(i+1)\delta-}\right]=n\bar
a_{j'}^{n,(i+1)\delta}
\end{align}
for $1\leq j\leq n$; and the condition variance and covariance satisfy
\begin{align}
&\tilde{\mathbb E}\left[\left(o_{l}^{n,(i+1)\delta}-n\bar a_l^{n,(i+1)\delta}\right)\left(o_{j}^{n,(i+1)\delta}-n\bar a_j^{n,(i+1)\delta}\right)\Big|\mathcal F_{(i+1)\delta-}\right]\nonumber\\
&=\left\{
\begin{array}{lll}
n\bar a_{j}^{n,(i+1)\delta}\left(1-\bar a_{j}^{n,(i+1)\delta}\right),\quad l=j\\
-n\bar a_{l}^{n,(i+1)\delta}\bar a_{j}^{n,(i+1)\delta},\qquad\quad\  l\neq j
\end{array}
\right.
\end{align}
for $1\leq l,j\leq n$.
\end{proposition}

The multinomial resampling algorithm essentially states that, at branching times, we sample $n$ times (with replacement) from the population of Gaussian random variables $X_j^n((i+1)\delta)$ (with means $v_j^n((i+1)\delta_-)$ and variances $\omega_j^n((i+1)\delta_-)),j=1,\ldots,n$ according to the multinomial probability distribution given by the corresponding normalised weights $\bar a_j^n((i+1)\delta_-),j=1,\ldots,n$. Therefore, by definition of multinomial distribution, $o_j^{n,(i+1)\delta}$ is the number of times $X_j^n((i+1)\delta)$ is chosen at time $(i+1)\delta$; that is to say, $o_j^{n,(i+1)\delta}$ is the number of offspring produced by this Gaussian random variable.

\section{Convergence Analysis}
In this section we deduce the evolution equation of the approximating measure $\rho^n$ for the generalised particle filters with Gaussian mixtures, and show its convergence to the target measure $\rho$ -- the solution of the Zakai equation, as well as the convergence of $\pi^n$ to $\pi$ -- the solution of the Kushner-Stratonovich equation. The correction mechanism for the generalised particle system involves either the use of the Tree Based Branching Algorithm (TBBA) or the multinomial resampling algorithm. These will be investigated in Sections 4.2 and 4.3 respectively. 

\subsection{Evolution Equation for $\rho^n$}
We firstly define the process $\xi^n=\{\xi_t^n;t\geq0\}$ by
$$
\xi_t^n\triangleq\left(\prod_{i=1}^{[t/\delta]}\frac{1}{n}\sum_{j=1}^na_j^{n,i\delta}\right)\left(\frac{1}{n}\sum_{j=1}^na_j^n(t)\right).
$$ 
Then $\xi^n$ is a martingale and by Exercise 9.10 in \cite{Bain and Crisan} we know for any $t\geq0$ and $p\geq1$, there exist two constants $c_1^{t,p}$ and $c_2^{t,p}$ which depend only on $t$, $p$, and $\max_{k=1,\ldots,m}\|h_k\|_{0,\infty}$, such that
\begin{equation}
\sup_{n\geq0}\sup_{s\in[0,t]}\tilde{\mathbb{E}}\left[(\xi_s^n)^p\right]\leq c_1^{t,p},\qquad\text{}
\end{equation} 
and
\begin{equation}\label{eq.bound_for_xi_2}
\max_{j=1,\ldots,n}\sup_{n\geq0}\sup_{s\in[0,t]}\tilde{\mathbb{E}}\left[(\xi_s^na_j^n(s))^p\right]\leq c_2^{t,p}.
\end{equation}
We use the martingale $\xi^n$ to linearise $\pi^n$ in order to make it easier to analyse its convergence. Let $\rho^n=\{\rho_t^n;t\geq0\}$ be the measure-valued process defined by
\begin{equation}\label{definition_of_rho}
\rho_t^n\triangleq\xi_t^n\pi_t^n=\frac{\xi_{[t/\delta]\delta}}{n}\sum_{j=1}^na_j^n(t)\Gamma_{v_j^n(t),\omega_j^n(t)},
\end{equation}
where $\Gamma_{v_j^n(t),\omega_j^n(t)}$ is the Gaussian measure with mean $v_j^n(t)$ and variance $\omega_j^n(t)$. We will show the convergence of $\rho^n$ to $\rho$ as the number of generalised particles $n$ increases.

The following proposition describes the evolution equation satisfied by the approximating sequence $\rho^n=\{\rho_t^n;t\geq0\}$ constructed using the algorithm described in the previous section. 
As discussed in Section 3, the approximation algorithm is constructed for the case where the state space of the signal process $X$ is $\mathbb R$. We adopt this assumption in this and the following sections. We first introduce the following notations:

\begin{align}
R_{s,j}^1(\varphi)=&\omega_j^n(s)\left[\frac{1}{2}(f\varphi''')(v_j^n(s))+\frac{\alpha}{4}(\sigma\varphi^{(4)})(v_j^n(s))+2\alpha\sigma^2(v_j^n(s))I_{4,j}^{(4)}(\varphi)-I_j(A\varphi)\right]\nonumber\\
+&(\omega_j^n(s))^2\left[f(v_j^n(s))I_{4,j}^{(5)}(\varphi)+\frac{\alpha\sigma^2(v_j^n(s))}{2\sqrt{\omega_j^n(s)}}I_{5,j}(\varphi)+\frac{1-\alpha}{2}\sigma^2(v_j^n(s))I_{4,j}^{(6)}(\varphi)\right],\label{eq.R^1} \\
R_{s,j}^2(\varphi)=&\omega_j^n(s)\left[\frac{1}{2}h(v_j^n(s))\varphi''(v_j^n(s))-I_j(h\varphi)\right]
+(\omega_j^n(s))^2h(v_j^n(s))I_{4,j}^{(4)}(\varphi),\label{eq.R^2} \\
R_{s,j}^3(\varphi)=&\sqrt{1-\alpha}\Bigg[\sigma(v_j^n(s))\varphi'(v_j^n(s))+\frac{1}{2}\omega_j^n(s)\sigma(v_j^n(s))\varphi'''(v_j^n(s))\nonumber\\
&\qquad\quad+(\omega_j^n(s))^2\sigma(v_j^n(s))I_{4,j}^{(5)}(\varphi)\Bigg]; \label{eq.R^3}
\end{align}
\begin{align}
I_{4,j}^{(k)}(\varphi)=&\int_{\mathbb{R}}\frac{y^4e^{\frac{-y^2}{2}}}{\sqrt{2\pi}}\int_0^1\varphi^{(k)}\left(v_j^n(s)+uy\sqrt{\omega_j^n(s)}\right)\frac{(1-u)^3}{6}dudy,\quad\text{for}\
k=4,5,6;\nonumber\\
I_{5,j}(\varphi)=&\int_{\mathbb{R}}\frac{y^5e^{\frac{-y^2}{2}}}{\sqrt{2\pi}}\int_0^1\varphi^{(5)}\left(v_j^n(s)+uy\sqrt{\omega_j^n(s)}\right)\frac{u(1-u)^3}{6}dudy;\nonumber\\
I_j(\psi)=&\int_{\mathbb{R}}\frac{y^2e^{\frac{-y^2}{2}}}{\sqrt{2\pi}}\int_0^1(\psi)''\left(v_j^n(s)+uy\sqrt{\omega_j^n(s)}\right)(1-u)dudy,\quad\text{for}\
\psi=A\varphi,h\varphi.\nonumber
\end{align}

\begin{proposition}\label{prop.equation_for_rho_t_n}
The measure-valued process $\rho^n=\{\rho_t^n:t\geq0\}$ satisfies the following evolution equation:
\begin{align}\label{eq.integral_form_of_rho_t_n_varphi}
\rho_t^n(\varphi)& = \rho_0^n(\varphi)+\int_0^t\rho_s^n(A\varphi)ds+\int_0^t\rho_s^n(h\varphi)dY_s+M_{[t/\delta]}^{n,\varphi}+B_t^{n,\varphi}
\end{align}
for any $\varphi\in C_b^m(\mathbb{R})$ and $t\in[0,T]$ with $m\geq6$.
In \eqref{eq.integral_form_of_rho_t_n_varphi},
$M^{n,\varphi}=\{M_i^{n,\varphi},i>0 \ and \ i\in\mathbb{N}\}$ is the discrete process
\begin{align}\label{eq.M_n_varphi}
&M_{[t/\delta]}^{n,\varphi}=\frac{1}{n}\sum_{i=0}^{[t/\delta]}\xi_{i\delta}^n\sum_{j=1}^n\Bigg[o_{j}^{n,i\delta}
\int_{\mathbb R}\varphi(x)\frac{e^{-\frac{(x-X_{j}^n(i\delta))^2}{2\alpha\beta}}}{\sqrt{2\pi\alpha\beta}}dx
-n\bar a_{j}^n(i\delta-)\int_{\mathbb{R}}\varphi(x)
\frac{e^{-\frac{(x-v_{j}^n(i\delta-))^2}{2\omega_{j}^n(i\delta-)}}}{\sqrt{2\pi\omega_{j}^n(i\delta-)}}dx\Bigg]
\end{align}
where $X_{j}^n(i\delta)\sim N(v_{j}^n(i\delta-),\omega_{j}^n(i\delta-))$ is a Gaussian random variable. Also in \eqref{eq.integral_form_of_rho_t_n_varphi},
$B_t^{n,\varphi}$ is the following process:
$$
B_t^{n,\varphi}=\frac{1}{n}\sum_{j=1}^n\int_{0}^{t}\xi_{[s/\delta]\delta}^na_j^n(s)\Big[R_{s,j}^1(\varphi)ds+R_{s,j}^2(\varphi)dY_s+R_{s,j}^3(\varphi)dV_s^{(j)}\Big].
$$
\end{proposition}
\begin{proof}
For any $\varphi\in C_b^m(\mathbb{R})$ and $t\in[i\delta,(i+1)\delta)$, we have from \eqref{definition_of_rho} that
\begin{align}\label{eq.rho_t_n_varphi}
\rho_t^n(\varphi)
&=\frac{\xi_{i\delta}^n}{n}\sum_{j=1}^na_j^n(t)\int_{\mathbb{R}}\varphi\left(v_j^n(t)+y\sqrt{\omega_j^n(t)}\right)\frac{1}{\sqrt{2\pi}}\exp\left(-\frac{y^2}{2}\right)dy,
\end{align}
with similar formulas for $A\varphi$ and $h\varphi$. We have the following Taylor expansions
\begin{align}\label{eq.taylor_expansion1}
\psi\left(v_j^n(t)+y\sqrt{\omega_j^n(t)}\right)&=\sum_{k=0}^{2p-1}
\frac{y^k}{k!}(\omega_j^n(t))^\frac{k}{2}\psi^{(k)}(v_j^n(t))\nonumber\\
&+y^{2p}\left(\omega_j^n(t)\right)^p\int_0^1\frac{1}{(2p)!}\psi^{(2p)}\left(v_j^n(t)+uy\sqrt{\omega_j^n(t)}\right)(1-u)^{2p-1}du,
\end{align}
where $\psi$ can be $\varphi$, $A\varphi$, or $h\varphi$.

\noindent By applying \eqref{eq.taylor_expansion1} (for $p=2$ and $p=1$) to \eqref{eq.rho_t_n_varphi} and the similar identities for $A\varphi$ and $h\varphi$, noting the fact that for any $k\geq1$ and $k\in\mathbb{N}$,
$$
\int_{\mathbb{R}}y^{2k-1}\frac{1}{\sqrt{2\pi}}\exp(-\frac{y^2}{2})dy=0,\quad
\int_{\mathbb{R}}y^{2k}\frac{1}{\sqrt{2\pi}}\exp(-\frac{y^2}{2})dy=\prod_{j=1}^k(2j-1),
$$
we obtain that
\begin{equation}\label{eq.taylored_rho_t_n_varphi}
\rho_t^n(\varphi)=\frac{\xi_{i\delta}^n}{n}\sum_{j=1}^na_j^n(t)\left[\varphi(v_j^n(t))+\frac{1}{2}\omega_j^n(t)\varphi''(v_j^n(t))\right]
+\frac{\xi_{i\delta}^n}{n}\sum_{j=1}^na_j^n(t)\left(\omega_j^n(t)\right)^2I_{4,j}^{(4)}(\varphi);
\end{equation}
\begin{equation}\label{eq.taylored_rho_t_n_Avarphi}
\rho_t^n(A\varphi)=\frac{\xi_{i\delta}^n}{n}\sum_{j=1}^na_j^n(t)\left[\left(A\varphi\right)(v_j^n(t))\right]
+\frac{\xi_{i\delta}^n}{n}\sum_{j=1}^na_j^n(t)\omega_j^n(t)I_j(A\varphi);
\end{equation}
\begin{equation}\label{eq.taylored_rho_t_n_hvarphi}
\rho_t^n(h\varphi)=\frac{\xi_{i\delta}^n}{n}\sum_{j=1}^na_j^n(t)\left[\left(h\varphi\right)(v_j^n(t))\right]
+\frac{\xi_{i\delta}^n}{n}\sum_{j=1}^na_j^n(t)\omega_j^n(t)I_j(h\varphi).
\end{equation}
Next we apply It\^o's formula to equation \eqref{eq.taylored_rho_t_n_varphi}, with the particles satisfying equations \eqref{eq.appro_with_variance}. After substituting \eqref{eq.taylored_rho_t_n_Avarphi} and \eqref{eq.taylored_rho_t_n_hvarphi}, we have, for $t\in[i\delta,(i+1)\delta)$,
\begin{align}\label{eq.rho_between_branches}
\rho_t^n(\varphi)& = \rho_{i\delta}^n(\varphi)+\int_{i\delta}^t\rho_s^n(A\varphi)ds+\int_{i\delta}^t\rho_s^n(h\varphi)dY_s\nonumber\\
&+\int_{i\delta}^{t}\frac{1}{n}\sum_{j=1}^n\xi_{i\delta}^na_j^n(s)\Big[R_{s,j}^1(\varphi)ds+R_{s,j}^2(\varphi)dY_s+R_{s,j}^3(\varphi)dV_s^{(j)}\Big].
\end{align}

Let $\mathcal{F}_{i\delta-}=\sigma\left(\mathcal{F}_s,0\leq s<i\delta\right)$ be the $\sigma$-algebra of the events up to time $i\delta$ (the time of the $i$-th-branching) and $\rho_{i\delta-}^n=\lim_{t\nearrow i\delta}\rho_t^n$. For any $t\geq0$, we have\footnote{We use the standard convention $\sum_{k=1}^0=0$.}
\begin{equation}
\rho_t^n(\varphi)=\rho_0^n(\varphi)+\sum_{i=1}^{[t/\delta]}(\rho_{i\delta}^n(\varphi)-\rho_{i\delta-}^n(\varphi))+\sum_{i=1}^{[t/\delta]}(\rho_{i\delta-}^n(\varphi)-\rho_{(i-1)\delta}^n(\varphi))+(\rho_t^n(\varphi)-\rho_{[t/\delta]\delta}^n(\varphi)),
\end{equation}
At the $i$-th correction event, each Gaussian measure is replaced by a random number $(o_{j}^{n,i\delta})$ of offsprings. Each offspring is a Gaussian measure with mean $X_j^n(i\delta)$ and variance $\alpha\beta$, where $X_j^n(i\delta)\sim\mathcal N(v_j^n(i\delta-),\omega_j^n(i\delta-))$. The weights of the offspring generalised particles are re-initialised to $1$, i.e. $a_{j}^n(i\delta)=1$; hence $\bar a_j^{n}(i\delta)=1/n$. So
$$
\pi_{i\delta}^n=\frac{1}{n}\sum_{j=1}^no_{j}^{n,k\delta}\Gamma_{X_{j}^n(k\delta),\alpha\beta},\qquad \text{and}\qquad
\pi_{i\delta}^n(\varphi)=\frac{1}{n}\sum_{j=1}^no_{j}^{n,k\delta}\int_{\mathbb R}\varphi(x)\frac{e^{-\frac{(x-X_{j}^n(i\delta))^2}{2\alpha\beta}}}{\sqrt{2\pi\alpha\beta}}dx.
$$
Before the correction event, we have
$$
\rho_{i\delta-}^n(\varphi)={\xi_{i\delta}^n}\sum_{j=1}^n\bar a_{j}^n(i\delta-)\int_{\mathbb{R}}\varphi\left(v_{j}^n(i\delta-)+y\sqrt{\omega_{j}^n(i\delta-)}\right)\frac{1}{\sqrt{2\pi}}\exp\left(-\frac{y^2}{2}\right)dy.
$$
We then obtain 
\begin{align}\label{eq.M_n_varphi}
M_{t/\delta}^{n,\varphi}\triangleq&\sum_{i=0}^{[t/\delta]}\left(\rho_{i\delta}^n(\varphi)-\rho_{i\delta-}^n(\varphi)\right)\nonumber\\
=&\frac{1}{n}\sum_{i=0}^{[t/\delta]}\xi_{i\delta}^n\sum_{j=1}^n\Bigg[o_{j}^{n,i\delta}\int_{\mathbb R}\varphi(x)\frac{e^{-\frac{(x-X_{j}^n(i\delta))^2}{2\alpha\beta}}}{\sqrt{2\pi\alpha\beta}}dx
-n\bar a_{j}^n(i\delta-)\int_{\mathbb{R}}\varphi(x)
\frac{e^{-\frac{(x-v_{j}^n(i\delta-))^2}{2\omega_{j}^n(i\delta-)}}}{\sqrt{2\pi\omega_{j}^n(i\delta-)}}dx\Bigg].
\end{align}
For $t\in[(i-1)\delta,i\delta)$, for $i=1,2,\ldots,[t/\delta]$,
using \eqref{eq.rho_between_branches}, we obtain that
\begin{align}\label{eq.rho_t_n_varphi_minus_M}
&\sum_{i=1}^{[t/\delta]}(\rho_{i\delta-}^n(\varphi)-\rho_{(i-1)\delta}^n(\varphi))+(\rho_t^n(\varphi)-\rho_{i\delta}^n(\varphi))\nonumber\\
=&\rho_0^n(\varphi)+\int_0^t\rho_s^n(A\varphi)ds+\int_0^t\rho_s^n(h\varphi)dY_s
+\int_{0}^{t}\frac{1}{n}\sum_{j=1}^n\xi_{[s/\delta]\delta}^na_j^n(s)\Big[R_{s,j}^1(\varphi)ds+R_{s,j}^2(\varphi)dY_s+R_{s,j}^3(\varphi)dV_s^{(j)}\Big],
\end{align}
Finally, \eqref{eq.M_n_varphi} and \eqref{eq.rho_t_n_varphi_minus_M} imply \eqref{eq.integral_form_of_rho_t_n_varphi}, which completes the proof.
\end{proof}

\begin{corollary}
Under the same assumption as in Proposition \ref{prop.equation_for_rho_t_n}, if we further assume that $\alpha=0$ in \eqref{eq.appro_with_variance}, we recover the classic particle filters with mixture of Dirac measures.
\end{corollary}

\subsection{Convergence Results for Generalised Particle Filters using the TBBA}
In order to investigate the convergence of the approximating measure $\rho^n$, we consider the mild form of the Zakai equation. One should note that the proof of the convergence in \cite{Bain and Crisan} using the dual, $\psi_s^{t,\varphi}$, of the measure-valued process $\rho$ does not work for our model. $\psi_s^{t,\varphi}$ is measurable with respect to the backward filtration $\mathcal{Y}_s^t=\sigma(Y_t-Y_r,\ r\in[s,t])$, and so is $R_{s,j}^2(\psi_s^{t,\varphi})$; however, the It\^o's integral $\int_0^tR_{s,j}^2(\psi_s^{t,\varphi})dY_s$ requires $R_{s,j}^2(\psi_s^{t,\varphi})$ is measurable with respect to the forward filtration $\mathcal{Y}_s=\sigma(Y_r,\ r\in[0,s])$. This leads to an anticipative integration which cannot be tackled in a standard manner.  Another approach is therefore required. Markov semigroups were used in \cite{Obanubi} to obtain relevant bounds on the error which in turn enables us to discuss the convergence rate. In the following this idea will be discussed in some details.

Let $(P_r)_{r\geq0}$ be the Markov semigroup whose infinitesimal generator is the operator $A$ and $\varphi$ is the single variable function which does not depend on $t$. 
Then from \eqref{eq.integral_form_of_rho_t_n_varphi} for $t\in[0,s]$, we get that
\begin{align}\label{eq.integral_form_of_rho_t_n_varphi_semigroup}
&\rho_t^n(P_{s-t}\varphi)=\rho_0^n(P_s\varphi)+\int_0^t\rho_r^n(hP_{s-r}\varphi)dY_r+M_{[t/\delta]}^{n,P\varphi}+B_t^{n,P\varphi}
\end{align}
and the error of the approximation has the representation
\begin{align}\label{eq.decomposition_of_the_difference}
&(\rho_t^n-\rho_{t})(P_{s-t}\varphi)
=(\rho_0^n-\rho_{0})(P_s\varphi)+\int_0^t(\rho_r^n-\rho_r)(hP_{s-r}\varphi)dY_r+M_{[t/\delta]}^{n,P\varphi}+B_t^{n,P\varphi},
\end{align}
where $M_i^{n,\varphi}$ and $B_t^{n,\varphi}$ are the same as in Proposition \ref{prop.equation_for_rho_t_n}, except that $\varphi$ replaced by $P_{s-r}\varphi$.

In order to prove the convergence of the approximating measures $\rho_t^n$ to the actual measure $\rho_t$, we need to control all the terms on the right hand side of \eqref{eq.decomposition_of_the_difference}. Now we will discuss each of them respectively in the following Lemmas. 
\begin{lemma}\label{lem.bound_for_initial_condition}
There exists a constant $c(p)$ independent of $n$ such that for any $p\geq1$ and $\varphi\in C_b(\mathbb{R})$, we have
\begin{equation*}
\tilde{\mathbb{E}}\left[(\rho_0^n(P_s\varphi)-\rho_0(P_s\varphi))^{p}\right]\leq\frac{c(p)}{n^{p/2}}\Vert\varphi\Vert^{p},\qquad t\in[0,T]
\end{equation*}
\end{lemma}
\begin{proof}
Note that $\rho_0^n(P_s\varphi)-\rho_0(P_s\varphi)=\pi_0^n(P_s\varphi)-\pi_0(P_s\varphi)$, and also note that 
$$\pi_0^n(P_s\varphi)-\pi_0(P_s\varphi)=\frac{1}{n}\sum_{j=1}^n\xi_j,$$
where $\xi_j\triangleq(P_s\varphi(v_j^n(0))-\pi_0(P_s\varphi),\ j=1,\ldots,n$ 
are independent identically distributed random variables with mean $0$, therefore
by the Marcinkiewicz-Zygmund inequality (see \cite{Ren and Liang}), there exists a constant $c=c(p)\leq(3\sqrt 2)^pp^{p/2}$ such that 
$$
\tilde{\mathbb E}\left[\left(\rho_0^n(P_s\varphi)-\rho_0(P_s\varphi)\right)^p\right]\leq\frac{c(p)}{n^{p/2}}\|\varphi\|^p,
$$
which completes the proof.
\end{proof}
\begin{lemma}\label{lem.bound_one}
For any $T>0$, there exists a constant $c_1(T)$ independent of $n$ such that for any $\varphi\in C_b^6(\mathbb{R})$,
\begin{equation*}
\tilde{\mathbb{E}}\left[\left(\frac{1}{n}\sum_{j=1}^n\int_{0}^{t}\xi_{[r/\delta]\delta}^na_j^n(r)R_{r,j}^1(P_{s-r}\varphi)dr\right)^2\right]\leq c_1(T)(\alpha\delta)^2\|\varphi\|_{6,\infty}^2.
\end{equation*}
\end{lemma}
\begin{proof}
For $\alpha\delta\leq1$ we have from \eqref{eq.R^1} that
$$
|R_{r,j}^1(P_{s-r}\varphi)|\leq C_{R^1}\alpha\delta\|\varphi\|_{6,\infty},
$$
where $C_{R^1}=C_{R^1}(f,\sigma,\varphi)$ is a constant depending on the upper bounds of $f$, $\sigma$, and $\varphi$. Then by Jensen's inequality, Fubini's theorem and \eqref{eq.bound_for_xi_2},
we have
\begin{align}
&\tilde{\mathbb{E}}\left[\left(\frac{1}{n}\sum_{j=1}^n\int_{0}^{t}\xi_{[r/\delta]\delta}^na_j^n(r)R_{r,j}^1(P_{s-r}\varphi)dr\right)^2\right]\nonumber\\
\leq&\frac{1}{n}\sum_{j=1}^ntC_{R^1}^2(\alpha\delta)^2\|\varphi\|_{6,\infty}^2\int_{0}^{t}\tilde{\mathbb E}\left[\left(\xi_{[r/\delta]\delta}^na_j^n(r)\right)^2\right]dr\nonumber
\leq T^2C_{R^1}^2c_2^{T,2}(\alpha\delta)^2\|\varphi\|_{6,\infty}^2.\nonumber
\end{align}
The result follows by letting $c_1(T)=T^2C_{R^1}^2c_2^{T,2}$.
\end{proof}
\begin{lemma}\label{lem.bound_two}
For any $T>0$, there exists a constant $c_2(T)$ independent of $n$ such that for any $\varphi\in C_b^4(\mathbb{R})$,
\begin{equation*}
\tilde{\mathbb{E}}\left[\left(\frac{1}{n}\sum_{j=1}^n\int_{0}^{t}\xi_{[r/\delta]\delta}^na_j^n(r)R_{r,j}^2(P_{s-r}\varphi)dY_r\right)^2\right]
\leq c_2(T)(\alpha\delta)^2\|\varphi\|_{4,\infty}^2.
\end{equation*}
\end{lemma}
\begin{proof}
From \eqref{eq.R^2}, we have for $\alpha\delta\leq1$
$$
R_{r,j}^2(P_{s-r}\varphi)\leq C_{R^2}\alpha\delta\|\varphi\|_{4,\infty},
$$
where $C_{R^2}=C_{R^2}(f,\sigma,h,\varphi)$ is a constant depending on the
upper bounds of $f$, $\sigma$, $h$, and $\varphi$. Then Burkholder-Davis-Gundy and Jensen's inequalities, Fubini's theorem, and \eqref{eq.bound_for_xi_2} yield
\begin{align}
&\tilde{\mathbb{E}}\left[\left(\frac{1}{n}\sum_{j=1}^n\int_{0}^{t}\xi_{[r/\delta]\delta}^na_j^n(r)R_{r,j}^2(P_{s-r}\varphi)dY_r\right)^2\right]\nonumber\\
\leq&\frac{1}{n^2}C_{R^2}^2\tilde C(\alpha\delta)^2\|\varphi\|_{4,\infty}^2\int_{0}^{t}\left[\sum_{j=1}^n\sum_{k=1}^n
\sqrt{\tilde{\mathbb E}\Big[(\xi_{[r/\delta]\delta}^na_j^n(r))^2\Big]\tilde{\mathbb E}\Big[(\xi_{[r/\delta]\delta}^na_k^n(r))^2\Big]}\right]dr\nonumber\\
\leq&TC_{R^2}^2\tilde Cc_2^{t,2}(\alpha\delta)^2\|\varphi\|_{4,\infty}^2,\nonumber
\end{align}
and the result follows by letting $c_2(T)=TC_{R^2}^2\tilde Cc_2^{t,2}$.
\end{proof}
\begin{lemma}\label{lem.bound_three}
For any $T>0$, there exists a constant $c_3(T)$ independent of $n$ such that for any $\varphi\in C_b^5(\mathbb{R})$,
\begin{equation*}
\tilde{\mathbb{E}}\left[\left(\frac{1}{n}\sum_{j=1}^n\int_{0}^{t}\xi_{[r/\delta]\delta}^na_j^n(r)R_{r,j}^3(P_{s-r}\varphi)dV_r^{(j)}\right)^2\right]
\leq\frac{c_3(T)}{n}\|\varphi\|_{5,\infty}^2.
\end{equation*}
\end{lemma}
\begin{proof}
From \eqref{eq.R^3}, we have for $\alpha\delta\leq1$
$$
R_{r,j}^3(P_{s-r}\varphi)\leq (\hat C_{R^3}+C_{R^3}\alpha\delta)\|\varphi\|_{5,\infty},
$$
where $\hat C_{R^3}$ and $C_{R^3}$ are constants depending on the
upper bounds of $f$, $\sigma$, and $\varphi$. Then by Burkholder-Davis-Gundy and Jensen's inequalities, Fubini's theorem, and \eqref{eq.bound_for_xi_2}, and noticing the fact that $\{V^{(j)}\}_{j=1}^n$ are mutually independent Brownian motions, we have
\begin{align}
&\tilde{\mathbb{E}}\left[\left(\frac{1}{n}\sum_{j=1}^n\int_{0}^{t}\xi_{[r/\delta]\delta}^na_j^n(r)R_{r,j}^3(P_{s-r}\varphi)dV_r^{(j)}\right)^2\right]\nonumber\\
\leq&\frac{1}{n^2}(\hat C_{R^3}+C_{R^3}\alpha\delta)^2\|\varphi\|_{5,\infty}^2\sum_{j=1}^n\tilde C\tilde{\mathbb{E}}\left[\left<\int_{0}^{t}\xi_{[r/\delta]\delta}^na_j^n(r)dV_r^{(j)}\right>_t\right]
\leq\frac{1}{n}T\tilde C(\hat C_{R^3}+C_{R^3}\alpha\delta)^2c_2^{t,2}\|\varphi\|_{5,\infty}^2,\nonumber
\end{align}
and the result follows by letting $c_3(T)=T\tilde C(\hat C_{R^3}+C_{R^3}\alpha\delta)^2c_2^{t,2}$.
\end{proof}

Recalling Proposition \ref{prop.equation_for_rho_t_n} and the semigroup operator
$P$, we can decompose $M^{n,\varphi}$ in the following way
\begin{align}
M_{[t/\delta]}^{n,\varphi}=&\frac{1}{n}\sum_{i=0}^{[t/\delta]}\xi_{i\delta}^n\sum_{j=1}^n\Bigg[o_{j}^{n,i\delta}\int_{\mathbb R}\varphi(x)\frac{e^{-\frac{(x-X_{j}^n(i\delta))^2}{2\alpha\beta}}}{\sqrt{2\pi\alpha\beta}}dx
-n\bar a_{j}^n(i\delta-)\int_{\mathbb{R}}\varphi(x)
\frac{e^{-\frac{(x-v_{j}^n(i\delta-))^2}{2\omega_{j}^n(i\delta-)}}}{\sqrt{2\pi\omega_{j}^n(i\delta-)}}dx\Bigg]\nonumber\\
\triangleq&A_{[t/\delta]}^{n,\varphi}+D_{[t/\delta]}^{n,\varphi}+G_{[t/\delta]}^{n,\varphi}\nonumber,
\end{align}
where $X_j^n(i\delta)\sim\mathcal N\left(v_j^n(i\delta-),\omega_j^n(i\delta-)\right)$
is a Gaussian distributed random variable, and
\begin{align}
A_{[t/\delta]}^{n,\varphi}&=\frac{1}{n}\sum_{i=1}^{[t/\delta]}\xi_{i\delta}^n\sum_{j=1}^n\left[\left(o_{j}^{n,i\delta}-n\bar
a_{j}^n(i\delta-)\right)\varphi(X_{j}^n(i\delta))\right];\\
D_{[t/\delta]}^{n,\varphi}&=\frac{1}{n}\sum_{i=1}^{[t/\delta]}\xi_{i\delta}^n\sum_{j=1}^n\left[o_{j}^{n,i\delta}\left(\int_{\mathbb
R}\varphi\left(X_j^n(i\delta)+y\sqrt{\alpha\beta}\right)\frac{e^{\frac{-y^2}{2}}}{\sqrt{2\pi}}dy-\varphi(X_{j}^n(i\delta))\right)\right];\\
G_{[t/\delta]}^{n,\varphi}&=\frac{1}{n}\sum_{i=1}^{[t/\delta]}\xi_{i\delta}^n\sum_{j=1}^nn\bar
a_{j}^n(i\delta-)\left[
\varphi(X_{j}^n(i\delta))-\int_{\mathbb{R}}\varphi\left(v_{j}^n(i\delta-)+y\sqrt{\omega_{j}^n(i\delta-)}\right)\frac{e^{-\frac{y^2}{2}}}{\sqrt{2\pi}}dy\right]\nonumber\\
&=\sum_{i=1}^{[t/\delta]}\sum_{j=1}^n\xi_{i\delta}^n\bar a_{j}^n(i\delta-)\left[\varphi(X_{j}^n(k\delta))-\tilde{\mathbb{E}}\left(\varphi(X_{j}^n(k\delta))\right)\right].
\end{align}
Then we have the following lemma:
\begin{lemma}\label{lem.bound_M}
For any $T>0$, 
and for any $\varphi\in C_b^2(\mathbb{R}),$ we have the following bounds for $A_{[t/\delta]}^{n,\varphi}$, $D_{[t/\delta]}^{n,\varphi}$ and $B_{[t/\delta]}^{n,\varphi}$:
\begin{align}
&\tilde{\mathbb{E}}\left[|A_{[t/\delta]}^{n,\varphi}|^2\right]\leq\frac{c_4(T)}{n\sqrt\delta}\|\varphi\|_{0,\infty}^2,\nonumber\\
&\tilde{\mathbb{E}}\left[|D_{[t/\delta]}^{n,\varphi}|^2\right]\leq c_{5}(T)(\alpha\beta)^2\|\varphi\|_{2,\infty}^2,\nonumber\\
&\tilde{\mathbb{E}}\left[\vert G_{[t/\delta]}^{n,\varphi}\vert^2\right]\leq\frac{ c_6(T)\alpha}{n}\|\varphi\|_{1,\infty}^2;
\end{align}
where $c_4(T)$, $c_5(T)$ and $c_6(T)$ are constants independent of $n$.
\end{lemma}
\begin{proof}
See Appendix A.1.
\end{proof}

The following Theorem, which is a variation of Theorem 8 in \cite{Obanubi}, establishes the convergence of finite \emph{signed} measure valued processes and allows us to use the bounds obtained from the above Lemmas to get the convergence results of $\rho_t^n$.  We denote by $a_s,\ a_{s,r}^k:\ C_b^m(\mathbb{R}^d)\rightarrow C_b^m(\mathbb{R}^d)$ two bounded linear operators with bounds $c$ and $C_k$ $(k=1,\ldots,\beta)$ respectively, i.e., $\Vert a_s(\varphi)\Vert_{m,\infty}\leq c\Vert\varphi\Vert_{m,\infty}$ and $\Vert a_{s,r}^k(\varphi)\Vert_{m,\infty}\leq C_k\Vert\varphi\Vert_{m,\infty}$.

\begin{theorem}\label{prop.general_convergence_result}
Let $\mu^n=\{\mu_t^n:\ t\geq0\}$ be a signed measure-valued process such that for any $\varphi\in C_b^m(\mathbb{R}^d)$, $m\geq6$, any fixed $\alpha\geq1$ and fixed $s>t$, we have
\begin{equation}
\mu_t^n(\varphi)=\mu_0^n(a_s(\varphi))+\sum_{l=1}^\alpha R_{t,l}^{n,\varphi}+\sum_{k=1}^\beta\int_0^t\mu_r^n(a_{s,r}^k(\varphi))dW_r^k,
\end{equation}
where $W=(W^k)_{k=1}^\beta$ is an $\beta$-dimensional Brownian motion. If for any $T>0$ there exist constants $\gamma_0,\gamma_1,\ldots,\gamma_\alpha$ such that for $t\in[0,T]$, $p\geq2$ and $q_l>0$ $(l=0,1,\ldots,\alpha)$,
\begin{equation}
\tilde{\mathbb{E}}\left[\vert\mu_0^n(a_s(\varphi))\vert^p\right]\leq\frac{\gamma_0}{n^{q_0}}\Vert\varphi\Vert_{m,\infty}^p,\qquad
\tilde{\mathbb{E}}\left[\vert R_{t,l}^{n,\varphi}\vert^p\right]\leq\frac{\gamma_l}{n^{q_l}}\Vert\varphi\Vert_{m,\infty}^p,\ l=1,\ldots,\alpha.
\end{equation}
Then for any $t\in[0,T]$, we have
\begin{equation}\label{eq.main_general_result}
\tilde{\mathbb{E}}\left[\vert\mu_t^n(\varphi)\vert^p\right]\leq\frac{c_t}{n^q}\Vert\varphi\Vert_{m,\infty}^p,
\end{equation}
where $c_t$ is a constant independent of $n$ and $q=\min(q_0,q_1,\ldots,q_\alpha)$.
\end{theorem}
\begin{proof}
See Appendix A.2.
\end{proof}

Applying the bounds in Lemmas \ref{lem.bound_for_initial_condition} to \ref{lem.bound_M}, one obtains the rate of convergence of the approximation in terms of the three parameters $n$, $\delta$ and $\alpha$, stated in the following theorem.
\begin{theorem}\label{firstthm.convergence_of_rho_t_n}
For any $T\geq0$, $m\geq6$, there exists a constant $c_7(T)$ independent of $n$, $\delta$ or $\alpha$ such that for any $\varphi\in C_b^m(\mathbb{R})$, we have for $t\in[0,T]$
\begin{equation}\label{feq.convergence_of_rho_t_n}
\tilde{\mathbb{E}}\left[(\rho_t^n(\varphi)-\rho_t(\varphi))^2
\right]\leq c_7(T)\Vert\varphi\Vert_{m,\infty}^2c(n,\delta,\alpha,\beta),
\end{equation}
where 
\[
c(n,\delta,\alpha,\beta)
=\max\left\{\frac{1}{n},\  (\alpha\delta)^2,\ \frac{1}{n\sqrt\delta},\ (\alpha\beta)^2,\ \frac{\alpha}{n}\right\}\nonumber.
\]
\end{theorem}

In what follows, we will discuss $c(n,\delta,\alpha,\beta)$ to obtain the $L^2$-convergence rate of the approximation process $\rho_t^n$.

When $\alpha=0$ in \eqref{eq.appro_with_variance}, the component Gaussian measures have null covariance matrices, in other words they are Dirac measures.  In this case $\rho_n$ is nothing other than the classic particle filter (see, for example, \cite{Bain and Crisan}). In this case several terms in $c(n,\delta,\alpha)$ coming from the covariance term disappear. The rate of convergence $c(n,\delta,0)$ becomes:  
$$
c(n,\delta,0)=\max\left\{\frac{1}{n},\ \frac{1}{n\sqrt\delta}\right\}.
$$
Obviously the fastest rate is obtained when $\delta$ is a fixed constant independent of $n$\(.\) The $L^2$-convergence rate will be in this case of order  $1/n$, which coincides with the results in \cite{Bain and Crisan}.

When $\alpha\in(0,1]$, the rate of convergence can deteriorate. First of all let us observe that we still need to choose  $\delta$ to be a fixed constant independent of $n$\(.\) Then the convergence depends on the simpler coefficient $c(n,\alpha)$ given by 
\[
c(n,\alpha,\beta)
=\max\left\{\frac{1}{n},\  \alpha^2,\ (\alpha\beta)^2,\ \frac{\alpha}{n}\right\}\nonumber
\]
In this case we need to choose $\alpha={1\over\sqrt n}$ (or of order $1/\sqrt n$) and $\beta$ to be a fixed constant independent of $n$ to ensure the optimal rate of convergence, which equals $1/n$. This discussion therefore leads to the following convergence result:
\begin{corollary}\label{thm.convergence_of_rho_t_n}
For any $T\geq0$, $m\geq6$, there exists constant $c_8(T)$ independent of $n$, such that for any $\varphi\in C_b^m(\mathbb{R})$, $t\in[0,T]$ and $\alpha\propto \frac{1}{\sqrt n}$ (defined in \eqref{eq.appro_with_variance}), we have
\begin{equation}\label{eq.convergence_of_rho_t_n}
\tilde{\mathbb{E}}\left[(\rho_t^n(\varphi)-\rho_t(\varphi))^2\right]\leq\frac{c_8(T)}{n}\Vert\varphi\Vert_{m,\infty}^2.
\end{equation}
\end{corollary}
For the normalised approximating measure $\pi^n$, we have the following main result.
\begin{theorem}\label{thm.convergence_of_pi}
For any $T\geq0$, $m\geq6$, there exists a constant $c_9(T)$ independent
of $n$ such that for $\alpha\propto\frac{1}{\sqrt n}$ and $\varphi\in C_b^m(\mathbb{R})$, we have
\begin{equation}
\tilde{\mathbb{E}}\left[\vert\pi_t^n(\varphi)-\pi_t(\varphi)\vert\right]\leq\frac{c_9(T)}{\sqrt n}\Vert\varphi\Vert_{m,\infty},\quad\quad
t\in[0,T].
\end{equation}
\end{theorem}
\begin{proof}
Observe that $\rho_t^n(\varphi)=\rho_t^n(\mathbf{1})\pi_t^n(\varphi)$, we then have
\begin{align}
&\pi_t^n(\varphi)-\pi_t(\varphi)
=\left(\rho_t^n(\varphi)-\rho_t(\varphi)\right)(\rho_t(\mathbf{1}))^{-1}-\pi_t^n(\varphi)
\left(\rho_t^n(\mathbf{1})-\rho_t(\mathbf{1})\right)(\rho_t(\mathbf{1}))^{-1}.\nonumber
\end{align}
Use the fact that $m_t\triangleq\sqrt{\tilde{\mathbb{E}}\left[(\rho_t(\mathbf{1}))^{-2}\right]}<\infty$
(see Exercise 9.16 of \cite{Bain and Crisan} for details), and by Cauchy-Schwartz
inequality:
\begin{align}\label{eq.difference_between_pi_t_n_and_pi_t}
&\tilde{\mathbb{E}}\left[\vert\pi_t^n(\varphi)-\pi_t(\varphi)\vert\right]\nonumber\\
\leq&m_t\left(\sqrt{\tilde{\mathbb{E}}\left[(\rho_t^n(\varphi)-\rho_t(\varphi))^2\right]}+\Vert\varphi\Vert_{0,\infty}\sqrt{\tilde{\mathbb{E}}\left[(\rho_t^n(\mathbf{1})-\rho_t(\mathbf{1}))^2\right]}\right),
\end{align}
and the result follows by applying Corollary \ref{thm.convergence_of_rho_t_n}
to the two expectations of the right hand side of \eqref{eq.difference_between_pi_t_n_and_pi_t}.
\end{proof}
A stronger convergence result for $\rho_t^n$ and $\pi_t^n$ will be proved in the following two propositions,
from which we can see that their convergence are uniform in time $t$. 
\begin{proposition}\label{thm.uniform_convergence_of_rho}
For any $T\geq0$, $m\geq6$, there exists a constant $c_{10}(T)$ independent
of $n$ such that for any $\varphi\in C_b^{m+2}(\mathbb{R})$,
\begin{equation}\label{eq.uniform_convergence_of_rho}
\tilde{\mathbb{E}}\left[\sup_{t\in[0,T]}(\rho_t^n(\varphi)-\rho_t(\varphi))^2\right]\leq\frac{c_{10}(T)}{n}\Vert\varphi\Vert_{m+2,\infty}^2.
\end{equation}
\end{proposition}
\begin{proof}
By Proposition \ref{prop.equation_for_rho_t_n} and the fact that $\rho_t(\varphi)$
satisfies Zakai equation, we have
\begin{align}\label{eq.rho_t_n_varphi_minus_rho_t_varphi}
\rho_t^n(\varphi)-\rho_t(\varphi)& =(\pi_0^n(\varphi)-\pi_0(\varphi))\nonumber\\
&+\int_0^t(\rho_s^n(A\varphi)-\rho_s(A\varphi))ds+ \int_0^t(\rho_s^n(h\varphi)-\rho_s(h\varphi))dY_s+M_{[t/\delta]}^{n,\varphi}\nonumber\\
&+\frac{1}{n}\sum_{j=1}^n\sum_{i=0}^{\infty}\int_{i\delta\wedge t}^{(i+1)\delta\wedge
t}\xi_{i\delta}^na_j^n(s)\Big[R_{s,j}^1(\varphi)ds+R_{s,j}^2(\varphi)dY_s+R_{s,j}^3(\varphi)dV_s^{(j)}\Big].
\end{align}
By  Lemmas \ref{lem.bound_one} -- \ref{lem.bound_three}, we know that,
\begin{equation*}
\tilde{\mathbb{E}}\left[\sup_{t\in[0,T]}\left(\frac{1}{n}\sum_{j=1}^n\sum_{i=0}^{\infty}\int_{i\delta\wedge
t}^{(i+1)\delta\wedge t}\xi_{i\delta}^na_j^n(r)R_{s,j}^1(\varphi)ds\right)^2\right]\leq
c_1(T)(\alpha\delta)^2\|\varphi\|_{6,\infty}^2;
\end{equation*}
\begin{equation*}
\tilde{\mathbb{E}}\left[\sup_{t\in[0,T]}\left(\frac{1}{n}\sum_{j=1}^n\sum_{i=0}^{\infty}\int_{i\delta\wedge
t}^{(i+1)\delta\wedge t}\xi_{i\delta}^na_j^n(r)R_{s,j}^2(\varphi)dY_s\right)^2\right]
\leq c_2(T)(\alpha\delta)^2\|\varphi\|_{4,\infty}^2;
\end{equation*}
\begin{equation*}
\tilde{\mathbb{E}}\left[\sup_{t\in[0,T]}\left(\frac{1}{n}\sum_{j=1}^n\sum_{i=0}^{\infty}\int_{i\delta\wedge
t}^{(i+1)\delta\wedge t}\xi_{i\delta}^na_j^n(r)R_{s,j}^3(\varphi)dV_s^{(j)}\right)^2\right]
\leq\frac{c_3(T)}{n}\|\varphi\|_{5,\infty}^2.
\end{equation*}
By Doob's maximal inequality and Lemma \ref{lem.bound_M}
\begin{equation*}
\tilde{\mathbb{E}}\left[\max_{i=1,\ldots,[T/\delta]}(M_i^{n,\varphi})^2\right]\leq4\tilde{\mathbb{E}}\left[\left(M_{[T/\delta]}^{n,\varphi}\right)^2\right]\leq\frac{4c_M([T/\delta])}{n}\Vert\varphi\Vert_{1,\infty}^2;
\end{equation*}

Now we only need to bound the first three terms on the right-hand side of
\eqref{eq.rho_t_n_varphi_minus_rho_t_varphi}. For the first term, using the
mutual independence of the initial locations of the particles $v_j^n(0)$,
\begin{equation*}
\tilde{\mathbb{E}}\left[\right(\pi_0^n(\varphi)-\pi_0(\varphi))^2]=\frac{1}{n}\left(\pi_0(\varphi^2)-\pi_0(\varphi)^2\right)\leq\frac{1}{n}\Vert\varphi\Vert_{2,\infty}^2.
\end{equation*} 
For the second term, by Jensen's inequality and Fubini's Theorem,
together with Corollary \ref{thm.convergence_of_rho_t_n}, we have
\begin{align}
& \tilde{\mathbb{E}}\left[\sup_{t\in[0,T]}\left(\int_0^t(\rho_s^n(A\varphi)-\rho_s(A\varphi))ds\right)^2\right]
\leq T\int_0^T\tilde{\mathbb{E}}\left[(\rho_s^n(A\varphi)-\rho_s(A\varphi))^2\right]ds\nonumber
\leq\frac{c_8(T)T^2}{n}\Vert\varphi\Vert_{m+2,\infty}^2.\nonumber
\end{align}
For the third term, similarly, by Burkholder-Davis-Gundy inequality and Fubini's
Theorem, together with Theorem \ref{thm.convergence_of_rho_t_n}, we have
\begin{align}
& \tilde{\mathbb{E}}\left[\sup_{t\in[0,T]}\left(\int_0^t(\rho_s^n(h\varphi)-\rho_s(h\varphi))dY_s\right)^2\right]\leq
\tilde{C}\tilde{\mathbb{E}}\left[\int_0^T(\rho_s^n(h\varphi)-\rho_s(h\varphi))^2ds\right]\nonumber
\leq\frac{\tilde{C}c_8(T)T\Vert
h\Vert_{0,\infty}^2}{n}\Vert\varphi\Vert_{m,\infty}^2.\nonumber
\end{align}
The above obtained  bounds together imply \eqref{eq.uniform_convergence_of_rho}.
\end{proof}
Similar to the proof of Theorem \ref{thm.convergence_of_pi}, we can show the following proposition.
\begin{proposition}\label{thm.uniform_convergence_of_pi}
For any $T\geq0$, $m\geq6$, there exists a constant $c_{11}(T)$ independent
of $n$ such that for and $\varphi\in C_b^{m+2}(\mathbb{R})$,
\begin{equation}
\tilde{\mathbb E}\left[\sup_{t\in[0,T]}\left\vert\pi_t^n(\varphi)-\pi_t(\varphi)\right\vert\right]\leq\frac{c_{11}(T)}{\sqrt n}\Vert\varphi\Vert_{m+2,\infty}.
\end{equation}
\end{proposition}

\begin{remark}The fact that the optimal value for $\alpha$ decreases with $n$ is not surprising. As the number of particles increases, the quantisation of the posterior distribution becomes finer and finer. Therefore, 
\emph{asymptotically}, the position and the weight of the particle provide  sufficient information to obtain a good approximation. In other words, asymptotically the classic particle filter is optimal.  
\end{remark}
\begin{remark}
Since the approximations $\rho_t^n$ and $\pi_t^n$ have smooth densities with respect to the Lebesgue measure, it makes it possible to study various properties of the density of $\rho_t$ from its approximation $\rho_t^n$ (for example, the position of their maximum value, the decay in time, the properties of their derivatives, etc).  This would be possible under the classic particle filtering framework, where the approximations are linear combinations of Dirac measures, only if a smoothing procedure is applied first (see \cite{Crisan and Miguez}).
\end{remark}

So far, the convergence results and $L^2$-error are obtained under probability $\tilde{\mathbb P}$; however, it is more natural to investigate these results under the original probability $\mathbb P$. The following proposition states the $L^2$-convergence result under $\mathbb P$.

\begin{proposition}
For any $T\geq0$, $m\geq6$, there exists
constant $c_{12}(T)$ independent of $n$, such that for any $\varphi\in C_b^m(\mathbb{R})$,
$t\in[0,T]$ and $\alpha\propto \frac{1}{\sqrt n}$ (defined in \eqref{eq.appro_with_variance}),
we have
\begin{equation}
{\mathbb{E}}\left[|\rho_t^n(\varphi)-\rho_t(\varphi)|\right]\leq\frac{c_{12}(T)}{\sqrt n}\Vert\varphi\Vert_{m,\infty}.
\end{equation}
\end{proposition}
\begin{proof}
Recalling the derivation of the new probability $\tilde{\mathbb P}$, we know that
\begin{equation}\label{eq.Z_tilde}
\tilde Z_t=\exp\left(\int_0^t h(X_s)dY_s-\frac{1}{2}\int_0^t
h^2(X_s)ds\right)\quad (t\geq0)
\end{equation}
is an  $\mathcal{F}_t$-adapted
martingale under $\tilde{\mathbb{P}}$ and
$$
\frac{d\mathbb{P}}{d\tilde{\mathbb{P}}}\Bigg\vert_{\mathcal{F}_t}=\tilde{Z}_t\quad
t\geq0.
$$
Therefore
\begin{align}
\mathbb E\left[|\rho_t^n(\varphi)-\rho_t(\varphi)|\right]=&\tilde{\mathbb E}\left[|\rho_t^n(\varphi)-\rho_t(\varphi)|\tilde Z_t\right]\nonumber
\leq\sqrt{\tilde{\mathbb E}\left[|\rho_t^n(\varphi)-\rho_t(\varphi)|^2\right]\tilde{\mathbb E}\left[(\tilde Z_t)^2\right]}\nonumber
\leq\sqrt{\frac{c_8(T)c_{\tilde z}}{n}}\|\varphi\|_{m,\infty}.
\end{align}
The result follows by letting $c_{12}(T)=\sqrt{c_8(T)c_{\tilde z}}$.
\end{proof}

\begin{remark}
If the correction mechanism is done using the Tree Based Branching Algorithm (TBBA), $L^p$-convergence of $\rho^n$ to $\rho$ cannot be generally obtained. This is because, in general, we do not have a control on the $p$th moment of $M_{[t/\delta]}^{n,\varphi}$ under $\tilde{\mathbb P}$. As a result, one can only obtain $L^1$-convergence result for $\rho^n$ under the original probability $\mathbb P$.
\end{remark}

\subsection{Convergence Results using the Multinomial Branching Algorithm}

In this section we show the convergence result for the case where the resampling is conducted by using Multinomial branching algorithm. The results in this section can be obtained in a similar manner as previous section, therefore the theorems in this section are only stated without proof. See \cite{Li} for detailed discussion and proofs of these results.

\begin{theorem}\label{thm.convergence_of_rho_t_n_multinomial}
For any $T\geq0$, $m\geq6$, there exist
constants $c_{13}(T)$ and $c_{14}(T)$ independent of $n$, such that for any $\varphi\in C_b^m(\mathbb{R})$,
$t\in[0,T]$ and $\alpha\propto \frac{1}{\sqrt n}$ (defined in \eqref{eq.appro_with_variance}),
we have
\begin{align}\label{eq.convergence_of_rho_t_n_multinomial}
\tilde{\mathbb{E}}\left[(\rho_t^n(\varphi)-\rho_t(\varphi))^2\right]&\leq\frac{c_{13}(T)}{n}\Vert\varphi\Vert_{m,\infty}^2;\\
\tilde{\mathbb{E}}\left[\vert\pi_t^n(\varphi)-\pi_t(\varphi)\vert\right]&\leq\frac{c_{14}(T)}{\sqrt
n}\Vert\varphi\Vert_{m,\infty}.
\end{align}
\end{theorem}

In contrast with the discussion under the TBBA, one can show that, if the correction is done using the multinomial branching algorithm, $L^p$-convergence result for $\rho^n$ and $\pi^n$ can be obtained for any $p\geq2$, namely we have the following theorem.

\begin{theorem}
For any $T\geq0$, $m\geq6$, there exists
constants $c_{15}(T)$ and $c_{16}(T)$ independent of $n$, such that for any $\varphi\in C_b^m(\mathbb{R})$,
$t\in[0,T]$ and $\alpha\propto \frac{1}{\sqrt n}$ (defined in \eqref{eq.appro_with_variance}),
we have
\begin{equation}
\tilde{\mathbb{E}}\left[|\rho_t^n(\varphi)-\rho_t(\varphi)|^p\right]\leq\frac{c_{15}(T)}{n^{p/2}}\Vert\varphi\Vert_{m,\infty}^{p/2};
\end{equation}
\begin{equation}
\tilde{\mathbb{E}}\left[|\pi_t^n(\varphi)-\pi_t(\varphi)|^p\right]\leq\frac{c_{16}(T)}{n^{p/2}}\Vert\varphi\Vert_{m,\infty}^{p/2}.
\end{equation}
\end{theorem}

\section{A Numerical Example}
In this section, we present some numerical experiments to test the performance of the approximations with mixture of Gaussian measures. The model chosen in this case is the Bene\v s filter. This is a stochastic filtering problem with nonlinear dynamics for the signal process and linear dynamics the observation process, with an analytical finite dimensional solution. The main reason for choosing this model is that it has a sufficient nonlinear behaviour to make it interesting, and more importantly, still has a closed form for its solution.
\subsection{The Model and its Exact Solution}
We assume that both the signal and the observation are one-dimensional. The dynamics of the signal $X$ is given by
\begin{align}
X_t=X_0+\int_0^tf(X_s)ds+\sigma V_t,
\end{align}
where $f(x)=\mu\sigma\tanh(\mu x/\sigma)$. We further assume that the observation $Y$ satisfies
\begin{align}
Y_t=\int_0^th(X_s)ds+W_t,
\end{align}
where $W$ is a standard Brownian motion independent of $V$, and $h(x)=h_1x+h_2$. We also assume that $X_0,\mu,h_1,h_2\in\mathbb R$ and $\sigma>0$.

Then from \cite{Crisan and Ortiz-Latorre} we know that the conditional law of $X_t$ given $\mathcal Y_t\triangleq\sigma(Y_s,0\leq s\leq t)$ has the exact expression of a weight mixture of two Gaussian distributions. In other words, the conditional distribution $\pi_t$ of $X_t$ is
$$
\pi_t= w^+\mathcal N(A^+/(2B_t),1/(2B_t))+w^-\mathcal N(A^-/(2B_t),1/(2B_t)),
$$
where $\mathcal N(\mu,\sigma^2)$ is the normal distribution with mean $\mu$ and variance $\sigma^2$, and
\begin{align}
w^\pm&\triangleq\exp\left((A_t^\pm)^2/(4B_t)\right)/\left[\exp\left((A_t^+)^2/(4B_t)\right)+\exp\left((A_t^-)^2/(4B_t)\right)\right]\nonumber\\
A_t^\pm&\triangleq\pm\frac{\mu}{\sigma}+h_1\Psi_t+\frac{h_1X_0+h_2}{\sigma\sinh(h_1\sigma t)}-\frac{h_2}{\sigma}\coth(h_1\sigma t)\nonumber\\
B_t&\triangleq\frac{h_1}{2\sigma}\coth(h_1\sigma t)\nonumber\\
\Psi_t&\triangleq\int_0^t\frac{\sinh(h_1\sigma s)}{\sinh(h_1\sigma t)}dY_s.\nonumber
\end{align}
We can, however, only observe $Y$ at a finite partition $\Pi^{m,T}=\{0=t_0<t_1<\cdots<t_{m-1}=T\}$ of $[0,T]$ in practice; thus we approximate the integral in the definition of $\Psi_t$ by
$$
\Psi_t\approx\sum_{k=0}^{i-1}\frac{\sinh(h_1\sigma t_{k+1})}{\sinh(h_1\sigma t_i)}(Y_{t_{k+1}}-Y_{t_k}),\quad\text{for}\ t_i\in\Pi^{n,T}.
$$
\subsection{Numerical Simulation Results}
We set values for the parameters $\mu$, $\sigma$, $h_1$, $h_2$, $X_0$ and $T$ as follows:
$$
\mu=0.3,\ h_1=0.8,\ h_2=0.0,\ \sigma=1.0,\ X_0=0.0,\ T=10.0;
$$ 
and then we compute one realisation for $X_t$ and one realisation for $Y_t$ respectively using the Euler scheme with an equidistant partition $\Pi^{m,T}=\{t_i=\frac{i}{m}T\}_{i=0,\ldots,m}$ with $m=10^6$. The realisation for $Y_t$ is then fixed and will act as the given observation path. All the simulations will be done assuming that we are given the previously obtained $Y_t$. With this previously simulated discrete path of $Y$, we can then approximate $\Psi_t$ and consequently compute the values of $A_t^\pm$, $B_t$ and $w_t^\pm$; so that we can compute the conditional law of $X_t$ given $\mathcal Y_t$. At the branching time, we use the Tree Based Branching Algorithm.
We will look at the convergence of the Gaussian mixture approximation and the classic particle filters in terms of the number of time steps in the partition and the number of particles respectively. 

We note that for the test function $\varphi(x)=x$, the Gaussian mixture approximation gives 
$$
\pi_t^n(\varphi)=\sum_{j=1}^n\bar a_j^n(t)\int_{\mathbb R}\left(v_j^n(t)+y\sqrt{\omega_j^n(t)}\right)\frac{1}{\sqrt{2\pi}}\exp\left(-\frac{y^2}{2}\right)dy
=\sum_{j=1}^n\bar a_j^n(t)v_j^n(t).
$$
This is almost the same result as the classic particle filters, except that the evolution equations satisfied by $v_j^n(t)$s are slightly different in two cases (see equation (9.4) in \cite{Bain and Crisan} and equation \eqref{eq.appro_with_variance} for details). It is therefore more interesting to look at $\pi_T(\varphi)$ for $\varphi(x)=x^2$ and $\varphi(x)=x^3$, that is, the second and third moments of the system at time $T$ given the observation $Y$ up to time $T$. To be specific, we estimate $\pi_T(\varphi)$ by $\pi_T^n(\varphi)$ with the number of particles (of Gaussian generalised particles) $n=40000$ and we choose various values for the number of time steps $m$ in the partition. We compute $\pi_T^n(\varphi)$ using classic particles and mixture of Gaussian measures respectively. Instead of the absolute error $\left|\pi_T^n(\varphi)-\pi_T(\varphi)\right|$, we consider the relative error 
$$\frac{\left|\pi_T^n(\varphi)-\pi_T(\varphi)\right|}{|\pi_T(\varphi)|}.$$

The convergence of both methods as the number of discretisation time steps $m$ increases can be seen from the following Figure \ref{fig.timesteps}, and for large number of time steps the Gaussian mixture approximation performs slightly better than the classic particle filters.\\

\begin{figure}[htb]
\centering
\includegraphics[width=150mm]{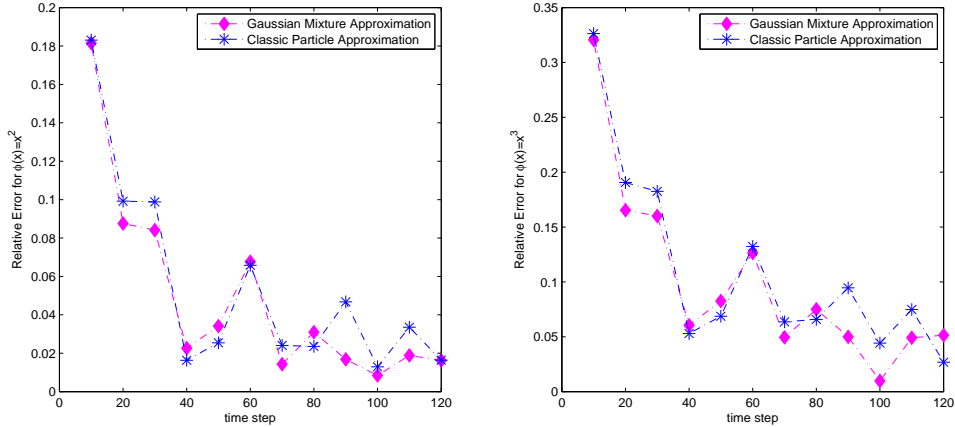}
\caption{Relative Errors with time steps for $\varphi(x)=x^2$ (left) \& $\varphi(x)=x^3$ (right)}
\label{fig.timesteps}
\end{figure}

In the following we fix the number of the discretisation time steps $m=100$ and vary the number of (generalised) particles $n$ in the approximating system. 

\begin{figure}[htb]
\centering
\includegraphics[width=150mm]{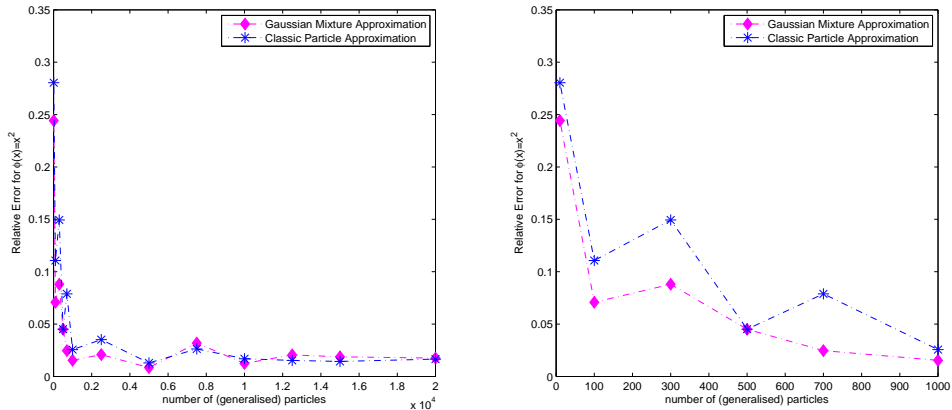}
\caption{Relative Errors with different number of (generalised) particles for $\varphi(x)=x^2$}
\label{fig.X2nParticles}
\end{figure}

\begin{figure}[htb]
\centering
\includegraphics[width=150mm]{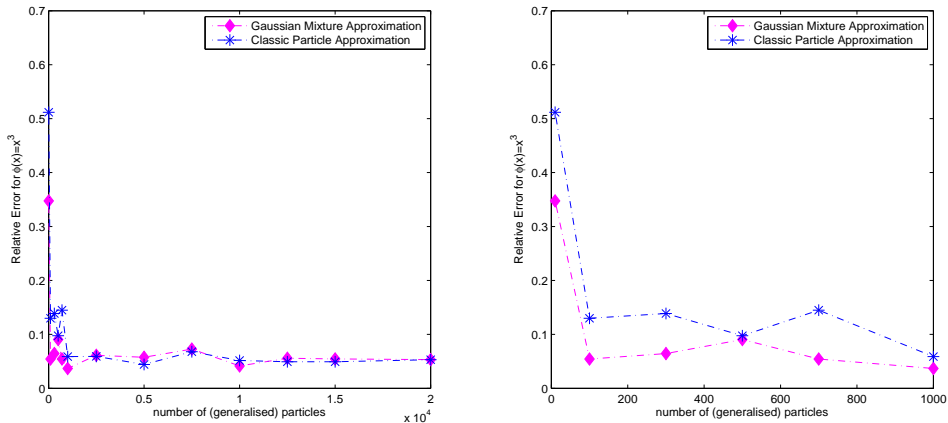}
\caption{Relative Errors with different number of (generalised) particles for $\varphi(x)=x^3$}
\label{fig.X3nParticles}
\end{figure}

From Figure \ref{fig.X2nParticles} and Figure \ref{fig.X3nParticles} we can see the convergence of both approximations with the increase of the number of (generalised) particles. It can be seen (from the right hand side of Figures \ref{fig.X2nParticles} and \ref{fig.X3nParticles}) that Gaussian mixture approximation performs better than the classic particle filters when the number of (generalised) particles $n$ is small. This is because by using the Gaussian mixture approximation, each (generalised) particle carries more information about the signal (from its variance) than the classic particle does. Therefore a smaller number of particles is required in order to obtain the same level of accuracy.

As the number of (generalised) particles increases, we can see (from the left hand side of Figures \ref{fig.X2nParticles} and \ref{fig.X3nParticles}) that the Gaussian mixture approximation converges faster than the classic particle filters; and we are able to obtain a good approximation for both methods with $10^4$ particles. There is no significant improvement if we increase the number $n$ of (generalised) particles further more in both approximating systems. It can also be seen that, although the error decreases as $n$ increases, it does not decrease monotonically and instead there is some fluctuation. This fluctuation comes from the additional randomness causing by the corrections or time discretisation step, which is not taken into account in the theoretical convergence analysis in this paper.

\section{Conclusions}
In this paper we analyse a class of approximations of the posterior
distribution under continuous time framework. In particular, we investigate in details the case where Gaussian mixtures are used to approximate the posterior distribution.

The $L^2$-convergence rate of such approximation is obtained. This method is a natural extension of the classic particle filters. In general, the approximating measure has a smooth density with respect to the Lebesgue measure and this can enable us to study more properties of the posterior measures than the classic particle filters do; especially this makes it possible to study various properties about the density of $\rho_t$ through its approximation $\rho_t^n$.  Furthermore, for a small number of particles, the Gaussian mixture particle filters also performs better.
It can also be seen that the asymptotic behaviour ($n\rightarrow\infty$) of the Gaussian mixtures approximation is similar to the classic particle filters, which is not surprising. As the number of (generalised) particles increases,
the quantisation of the posterior distribution becomes finer
and finer. Therefore, \emph{asymptotically}, the positions and the
weights of the particles provide sufficient information to obtain
a good approximation.

Apart from the $L^2$-convergence rate, a central limit type result can also be obtained for such Gaussian mixtures approximations, which can be found in \cite{Li} with a comprehensive study. Another paper containing a comprehensive study of the central limit theorem is in preparation (see \cite{Crisan and Li 2}).

\appendix
\addappheadtotoc

\section{Appendix}\label{ch.Multinomial_convergence}

\subsection{Proof of Lemma \ref{lem.bound_M}}
Without loss of generality, we choose the test function $\varphi$ to be non-negative (since we can write $\varphi=\varphi^+-\varphi^-$). Since the random variables $\{o_{j'}^{n,k\delta},\ j'=1,\ldots,n\}$ are negatively correlated (see Proposition 9.3 in \cite{Bain and Crisan}), it follows that
\begin{align}
&\tilde{\mathbb E}\left[\left(\frac{1}{n}\sum_{j=1}^n\xi_{i\delta}^n\left(\left(o_{j}^{n,i\delta}-n\bar a_{j}^n(i\delta-)\right)\varphi(X_{j}^n(i\delta))\right)\right)^2\Bigg|\mathcal Y_{i\delta-}\vee\mathcal F_{i\delta-}\right]\nonumber\\
\leq&\frac{(\xi_{i\delta}^n)^2}{n^2}\|\varphi\|_{0,\infty}^2\sum_{j=1}^n\left\{n\bar a_j^n(i\delta-)\right\}\left\{1-n\bar a_j^n(i\delta-)\right\}\nonumber\\
\leq&\frac{(\xi_{i\delta}^n)^2}{n^2}\|\varphi\|_{0,\infty}^2\sum_{j=1}^n\left|1-n\bar a_j^n(i\delta-)\right|\nonumber\\
\leq&\frac{(\xi_{i\delta}^n)^2}{n^2}\|\varphi\|_{0,\infty}^2nC_\delta\sqrt\delta
=C_\delta\sqrt\delta\frac{(\xi_{i\delta}^n)^2}{n}\|\varphi\|_{0,\infty}^2;\nonumber
\end{align}
By taking the expectation on both sides, we have
\begin{align}
\tilde{\mathbb E}\left[\left(\frac{1}{n}\sum_{j=1}^n\xi_{i\delta}^n\left(\left(o_{j}^{n,i\delta}-n\bar a_{j}^n(i\delta-)\right)\varphi(X_{j}^n(i\delta))\right)\right)^2\right]\leq\frac{C_\delta c_1^{t,2}}{n}\sqrt\delta\|\varphi\|_{0,\infty}^2.\nonumber
\end{align}
Therefore
\begin{align}
\tilde{\mathbb E}\left[\left(A_{[t/\delta]}^{n,\varphi}\right)^2\right]\leq\sum_{i=1}^{[t/\delta]}\frac{C_\delta c_1^{t,2}}{n}\sqrt\delta\|\varphi\|_{0,\infty}^2\leq\frac{[t/\delta]C_\delta c_1^{t,2}}{n}\sqrt\delta\|\varphi\|_{0,\infty}^2
\leq\frac{tC_\delta c_1^{t,2}}{\sqrt\delta n}\|\varphi\|_{0,\infty}^2
\end{align}
if we let $c_4(T)=TC_\delta c_1^{T,2}$.

For $D_{[t/\delta]}^{n,\varphi}$, first by noting that
\begin{align}
&\int_{\mathbb
R}\varphi\left(X_j^n(i\delta)+y\sqrt{\alpha\beta}\right)\frac{e^{\frac{-y^2}{2}}}{\sqrt{2\pi}}dy-\varphi(X_{j}^n(i\delta))
=\frac{\alpha\beta}{2}\varphi''\left(X_j^n(i\delta)\right)+\mathcal{O}\left((\alpha\beta)^2\right);
\end{align}
then it is clear that we only need to show
\begin{align}
\tilde{\mathbb E}\left[\left(\frac{1}{n}\sum_{i=1}^{[t/\delta]}\xi_{i\delta}^n\sum_{j=1}^no_{j}^{n,i\delta}\left((\alpha\beta)\varphi''\left(X_j^n(i\delta)\right)\right)\right)^2\right]\leq{c_5(T)(\alpha\beta)^2}\|\varphi\|_{0,\infty}^2.
\end{align}
Observe that
\begin{align}\label{eq.bound_for_one_time_period}
&\tilde{\mathbb E}\left[\left(\frac{1}{n}\xi_{i\delta}^n\sum_{j=1}^no_{j}^{n,i\delta}\left((\alpha\beta)\varphi''\left(X_j^n(i\delta)\right)\right)\right)^2\Bigg|\mathcal F_{i\delta-}\right]\nonumber\\
\leq&\frac{(\xi_{i\delta}^n)^2(\alpha\beta)^2}{n^2}\|\varphi\|_{2,\infty}^2\tilde{\mathbb
E}\left[\left(\sum_{j=1}^n\left[\left(o_{j}^{n,i\delta}-n\bar a_j^n(i\delta-)\right)+n\bar
a_j^n(i\delta-)\right]\right)^2\Bigg|\mathcal F_{i\delta-}\right]\nonumber\\
\leq&\frac{2(\xi_{i\delta}^n)^2(\alpha\beta)^2}{n^{2}}\|\varphi\|_{2,\infty}^2\left[\sum_{j=1}^n\tilde{\mathbb
E}\left[\left(o_{j}^{n,i\delta}-n\bar a_j^n(i\delta-)\right)^2\Big|\mathcal
F_{i\delta-}\right]+\left(\sum_{j=1}^nn\bar a_j^n(i\delta-)\right)^2\right].
\end{align}
For Tree Based Branching Algorithm (TBBA), by Proposition \ref{prop.tbba_proposition},
\begin{align}
\tilde{\mathbb E}\left[\left(o_{j}^{n,i\delta}-n\bar a_j^n(i\delta-)\right)^2\Big|\mathcal
F_{i\delta-}\right]\leq\left\{n\bar a_j^n(i\delta-)\right\}\left(1-\left\{n\bar
a_j^n(i\delta-)\right\}\right)\leq\frac{1}{4};
\end{align}
and then by taking expectation on both sides of \eqref{eq.bound_for_one_time_period},
we have
\begin{align}
&\tilde{\mathbb E}\left[\left(\frac{1}{n}\xi_{i\delta}^n\sum_{j=1}^no_{j}^{n,i\delta}\left((\alpha\beta)\varphi''\left(X_j^n(i\delta)\right)\right)\right)^2\right]\nonumber\\
\leq&\frac{2\|\varphi\|_{2,\infty}^2(\alpha\beta)^2}{n^{2}}\frac{n}{4}+\frac{2\|\varphi\|_{2,\infty}^2(\alpha\beta)^2}{n^2}n^2\tilde{\mathbb
E}\left[(\xi_{i\delta}^n)^2\left(\sum_{j=1}^n\bar a_j^n(i\delta-)\right)^2\right]\nonumber\\
\leq&\frac{\|\varphi\|_{2,\infty}^2(\alpha\beta)^2}{2n}+{2\|\varphi\|_{2,\infty}^2(\alpha\beta)^2}\sqrt{\tilde{\mathbb
E}\left[(\xi_{i\delta}^n)^4\right]\sum_{l=1}^n\sum_{j=1}^n\sqrt{\tilde{\mathbb
E}\left[\bar a_l^n(i\delta-)^2\right]\tilde{\mathbb E}\left[\bar a_j^n(i\delta-)^2\right]}}\nonumber\\
\leq&\frac{\|\varphi\|_{2,\infty}^2(\alpha\beta)^2}{2n}+{2\|\varphi\|_{2,\infty}^2(\alpha\beta)^2}\sqrt{c_1^{t,4}e^{c_2t}}\nonumber\\
\leq&{c^T(\alpha\beta)^2}\|\varphi\|_{2,\infty}^2,
\end{align}
where $c^T=\frac{1}{2}+2\sqrt{c_1^{t,4}e^{c_2t}}$.

Therefore, we obtain
\begin{align}
\tilde{\mathbb E}\left[\left(\frac{1}{n}\sum_{i=1}^{[t/\delta]}\xi_{i\delta}^n\sum_{j=1}^no_{j}^{n,i\delta}\left((\alpha\beta)\varphi''\left(X_j^n(i\delta)\right)\right)\right)^2\right]\nonumber
\leq[t/\delta]\sum_{i=1}^{[t/\delta]}{c^T(\alpha\beta)^2}\|\varphi\|_{2,\infty}^2={c_5(T)(\alpha\beta)^2}\|\varphi\|_{2,\infty}^2
\end{align}
if we let $c_5(T)=T^2c^T/\delta^2$.\\

As for $G_{[t/\delta]}^{n,\varphi}$, first note that $X_{j}^n(i\delta)\sim N\left(v_{j}^n(i\delta),\omega_{j}^n(i\delta)\right)$ and $X_{j}^n$s are mutually independent $(j=1,\ldots,n)$, then we have
\begin{align}
&\tilde{\mathbb E}\left[\left(\sum_{j=1}^n\xi_{i\delta}^n\bar a_{j}^n(i\delta-)\left[\varphi(X_{j}^n(k\delta))-\tilde{\mathbb{E}}\left(\varphi(X_{j}^n(k\delta))\right)\right]\right)^2\Bigg|\mathcal Y_{i\delta-}\right]\nonumber\\
\leq&\sum_{j=1}^n\left(\xi_{i\delta}^n\bar a_{j}^n(i\delta-)\right)^24\|\varphi'\|_{0,\infty}^2\tilde{\mathbb E}\left[\left(X_{j}^n(k\delta)-\tilde{\mathbb{E}}\left(X_{j}^n(k\delta)\right)\right)^2\Bigg|\mathcal Y_{i\delta-}\right]\nonumber\\
\leq&4\sum_{j=1}^n\left(\xi_{i\delta}^n\bar a_{j}^n(i\delta-)\right)^2\|\varphi\|_{1,\infty}^2\|\sigma\|_{0,\infty}^2\alpha\delta
\triangleq c_{\sigma}\alpha\delta\|\varphi\|_{1,\infty}^2\sum_{j=1}^n\left(\xi_{i\delta}^n\bar a_{j}^n(i\delta-)\right)^2.\nonumber
\end{align}
We know from the proof of Lemma 4.4 in \cite{Obanubi} that for any $p>0$,
$
\tilde{\mathbb E}\left[\left(\bar a_j^n(t)\right)^p\right]\leq\frac{1}{n^p}\exp(c_pt);
$
then by taking the expectation on both sides, we have
\begin{align}
&\tilde{\mathbb E}\left[\left(\sum_{j=1}^n\xi_{i\delta}^n\bar a_{j}^n(i\delta-)\left[\varphi(X_{j}^n(k\delta))-\tilde{\mathbb{E}}\left(\varphi(X_{j}^n(k\delta))\right)\right]\right)^2\right]\nonumber\\
\leq&c_{\sigma}\alpha\delta\|\varphi\|_{1,\infty}^2\sum_{j=1}^n\tilde{\mathbb E}\left[\left(\xi_{i\delta}^n\bar a_{j}^n(i\delta-)\right)^2\right]
\leq c_{\sigma}\alpha\delta\|\varphi\|_{1,\infty}^2\sum_{j=1}^n\sqrt{\tilde{\mathbb E}\left[\left(\xi_{i\delta}^n\right)^4\right]\tilde{\mathbb E}\left[\left(\bar a_{j}^n(i\delta-)\right)^4\right]}\nonumber\\
\leq&c_{\sigma}\alpha\delta\|\varphi\|_{1,\infty}^2\frac{1}{n^2}\sum_{j=1}^n\sqrt{c_1^{t,4}\exp(c_4t)}
=\frac{c_{\sigma}\sqrt{c_1^{t,4}\exp(c_4t)}}{n}\alpha\delta\|\varphi\|_{1,\infty}^2.\nonumber
\end{align}
Finally we have
\begin{align}\label{eq.bound_for_the_second_term_in_M}
\tilde{\mathbb E}\left[\left(G_{[t/\delta]}^{n,\varphi}\right)^2\right]\leq\sum_{i=1}^{[t/\delta]}\frac{c_{\sigma}\sqrt{c_1^{t,4}\exp(c_4t)}}{n}\alpha\delta\|\varphi\|_{1,\infty}^2
\leq \frac{tc_{\sigma}\sqrt{c_1^{t,4}\exp(c_4t)}}{n}\alpha\|\varphi\|_{1,\infty}^2.
\end{align}
The result follows by letting $c_6(T)=Tc_{\sigma}\sqrt{c_1^{T,4}\exp(c_4T)}$.

\subsection{Proof of Theorem \ref{prop.general_convergence_result}}
The proof is standard and we present here its principle steps, further details can be found in Section A.2 in \cite{Li}.

We first show that for any $t\in[0,T]$
$
\Vert\mu_t^n(\mathbf{1})\Vert_p^p=\tilde{\mathbb{E}}\left[\vert\mu_t^n(\mathbf{1})\vert^p\right]<\infty.
$
Observe that for $\varphi=\mathbf{1}$ and $t\in[0,T]$
\begin{equation*}
\mu_t^n(\mathbf{1})=\mu_0^n(a_s(\mathbf{1}))+\sum_{l=1}^\alpha R_{t,l}^{n,\mathbf{1}}+\sum_{k=1}^\beta\int_0^t\mu_r^n\left(a_{s,r}^k(\mathbf{1})\right)dW_r^k,
\end{equation*}
then, by Minkowski, Burkholder-Davis-Gundy and Jensen's inequalities we have that
\begin{align}\label{eq.control_of_mu_t_n_1}
\Vert\mu_t^n(\mathbf{1})\Vert_p^p
&\leq2^{p-1}(\alpha+1)^p\frac{\gamma}{n^q}+2^{p-1}\beta^pKt^{p/2-1}C^p\int_0^t\Vert\mu_r^n\left(\mathbf{1}\right)\Vert_p^pdr,
\end{align}
where $C=\max(C_1,\ldots,C_\beta)$. Then from Gronwall's inequality we have
\begin{align}
&\tilde{\mathbb{E}}\left[\left|\mu_s^n\left(\mathbf{1}\right)\right|^p\right]=\Vert\mu_s^n\left(\mathbf{1}\right)\Vert_p^p\nonumber
\leq D<\infty.\nonumber
\end{align}
Using a similar approach, with $\mathbf{1}$ replaced by $\varphi$, we obtain
\begin{equation*}
\Vert\mu_t^n(\varphi)\Vert_p^p\leq2^{p-1}(\alpha+1)^p\frac{\gamma}{n^q}\Vert\varphi\Vert_{m,\infty}^p+2^{p-1}\beta^{p-1}Kt^{p/2-1}\sum_{k=1}^\beta \int_0^t\tilde{\mathbb{E}}\left[\left|\mu_r^n\left(a_{s,r}^k(\varphi)\right)\right|^p\right]dr.
\end{equation*}
Now denote by
\begin{equation}\label{eq.A_s_t_k_1}
A_{s,t}^k\triangleq\int_0^t\tilde{\mathbb{E}}\left[\left|\mu_r^n\left(a_{s,r}^k(\varphi)\right)\right|^p\right]dr=\int_0^t\Vert\mu_r^n\left(a_{s,r}^k(\varphi)\right)\Vert_p^pdr,
\end{equation}
and $\Delta=(\alpha+1)^p\frac{\gamma}{n^q}$, we have
\begin{equation}\label{eq.control_of_mu_t_n_varphi}
\Vert\mu_t^n(\varphi)\Vert_p^p\leq2^{p-1}\Delta\Vert\varphi\Vert_{m,\infty}^p+2^{p-1}\beta^{p-1}Kt^{p/2-1}\sum_{k=1}^\beta A_{s,t}^k.
\end{equation}
Similarly, we have that
\begin{align}\label{eq.first_iteration}
\Vert\mu_t^n(\varphi)\Vert_p^p\leq
&2^{p-1}\Delta\Vert\varphi\Vert_{m,\infty}^p+2^{p-1}\beta^pKt^{p/2}C^pD\Vert\varphi\Vert_{m,\infty}^p.
\end{align}
Replacing $\varphi$ by $a_{s,r}^k(\varphi)$ in \eqref{eq.first_iteration}, we get that
\begin{align}\label{eq.first_bound}
\Vert\mu_r^n\left(a_{s,r}^k(\varphi)\right)\Vert_p^p\leq&2^{p-1}\Delta\Vert a_{s,r}^k(\varphi)\Vert_{m,\infty}^p+2^{p-1}\beta^pKr^{p/2}C^pD\Vert a_{s,r}^k(\varphi)\Vert_{m,\infty}^p\nonumber\\
\leq&2^{p-1}\Delta C^p\Vert\varphi\Vert_{m,\infty}^p+2^{p-1}\beta^pKr^{p/2}C^{2p}D\Vert\varphi\Vert_{m,\infty}^p,
\end{align}
Substituting into \eqref{eq.A_s_t_k_1} and denote by $\kappa=p/2$, we have for $k=1,\ldots,\beta$
\begin{equation}\label{eq.A_s_t_k_2}
A_{s,t}^k\leq2^{p-1}\Delta C^pt\Vert\varphi\Vert_{m,\infty}^p+2^{p-1}\beta^pKC^{2p}D\frac{t^{\kappa+1}}{\kappa+1}\Vert\varphi\Vert_{m,\infty}^p
\end{equation}
and \eqref{eq.control_of_mu_t_n_varphi} becomes
\begin{align}\label{eq.second_iteration}
\Vert\mu_t^n(\varphi)\Vert_p^p
\leq&2^{p-1}\Delta\Vert\varphi\Vert_{m,\infty}^p+2^{2(p-1)}\beta^pK C^pt^\kappa\Delta\Vert\varphi\Vert_{m,\infty}^p+2^{2(p-1)}\beta^{2p}K^2C^{2p}D\frac{t^{2\kappa}}{\kappa+1}\Vert\varphi\Vert_{m,\infty}^p.
\end{align}
Repeat what was done in \eqref{eq.first_bound} and \eqref{eq.A_s_t_k_2}, and from \eqref{eq.second_iteration}, we have that
\begin{align}
A_{s,t}^k\leq&2^{p-1}C^p\Delta t\Vert\varphi\Vert_{m,\infty}^p+2^{2(p-1)}\beta^pKC^{2p}\Delta\frac{t^{\kappa+1}}{\kappa+1}\Vert\varphi\Vert_{m,\infty}^p\nonumber\\
+&2^{2(p-1)}\beta^{2p}K^2C^{3p}D\frac{t^{2\kappa+1}}{(\kappa+1)(2\kappa+1)}\Vert\varphi\Vert_{m,\infty}^p;\nonumber
\end{align}
and then \eqref{eq.control_of_mu_t_n_varphi} becomes
\begin{align}
\Vert\mu_t^n(\varphi)\Vert_p^p\leq&2^{p-1}\Delta\Vert\varphi\Vert_{m,\infty}^p+2^{2(p-1)}\beta^pK C^pt^\kappa\Delta\Vert\varphi\Vert_{m,\infty}^p+2^{3(p-1)}\beta^{2p}K^2C^{2p}\frac{t^{2\kappa}}{\kappa+1}\Delta\Vert\varphi\Vert_{m,\infty}^p\nonumber\\
+&2^{3(p-1)}\beta^{3p}K^3C^{3p}D\frac{t^{3\kappa}}{(\kappa+1)(2\kappa+1)}\Vert\varphi\Vert_{m,\infty}^p.\nonumber
\end{align}
Repeat the iteration process again, we have that
\begin{align}
A_{s,t}^k\leq&2^{p-1}C^p\Delta t\Vert\varphi\Vert_{m,\infty}^p+2^{2(p-1)}\beta^pKC^{2p}\Delta\frac{t^{\kappa+1}}{\kappa+1}\Vert\varphi\Vert_{m,\infty}^p\nonumber\\
+&2^{3(p-1)}\beta^{2p}K^2C^{3p}\Delta\frac{t^{2\kappa+1}}{(\kappa+1)(2\kappa+1)}\Vert\varphi\Vert_{m,\infty}^p\nonumber\\
+&2^{3(p-1)}\beta^{3p}K^3C^{4p}D\frac{t^{3\kappa+1}}{(\kappa+1)(2\kappa+1)(3\kappa+1)}\Vert\varphi\Vert_{m,\infty}^p;\nonumber
\end{align}
and that
\begin{align}
\Vert\mu_t^n(\varphi)\Vert_p^p\leq&2^{p-1}\Delta\Vert\varphi\Vert_{m,\infty}^p+2^{2(p-1)}\beta^pK C^pt^\kappa\Delta\Vert\varphi\Vert_{m,\infty}^p+2^{3(p-1)}\beta^{2p}K^2C^{2p}\frac{t^{2\kappa}}{\kappa+1}\Delta\Vert\varphi\Vert_{m,\infty}^p\nonumber\\
+&2^{4(p-1)}\beta^{3p}K^3C^{3p}\frac{t^{3\kappa}}{(\kappa+1)(2\kappa+1)}\Delta\Vert\varphi\Vert_{m,\infty}^p\nonumber\\
+&2^{4(p-1)}\beta^{4p}K^4C^{4p}D\frac{t^{4\kappa}}{(\kappa+1)(2\kappa+1)(3\kappa+1)}\Vert\varphi\Vert_{m,\infty}^p.\nonumber
\end{align}
In general after $k^{th}-$iteration, we have that
\begin{align}
\Vert\mu_t^n(\varphi)\Vert_p^{p,k}\triangleq&\Vert\mu_t^n(\varphi)\Vert_p^p\nonumber\\
\leq&2^{p-1}\Delta\Vert\varphi\Vert_{m,\infty}^p+2^{2(p-1)}\beta^pK C^pt^\kappa\Delta\Vert\varphi\Vert_{m,\infty}^p+2^{3(p-1)}\beta^{2p}K^2C^{2p}\frac{t^{2\kappa}}{\kappa+1}\Delta\Vert\varphi\Vert_{m,\infty}^p\nonumber\\
+&\cdots+\frac{2^{k(p-1)}\beta^{(k-1)p}K^{k-1}C^{(k-1)p}t^{(k-1)\kappa}}{(\kappa+1)(2\kappa+1)(3\kappa+1)\cdots((k-2)\kappa+1)}\Delta\Vert\varphi\Vert_{m,\infty}^p\nonumber\\
+&2^{k(p-1)}\beta^{kp}K^kC^{kp}D\frac{t^{r\kappa}}{(\kappa+1)(2\kappa+1)\cdots((k-2)\kappa+1)((k-1)\kappa+1)}\Vert\varphi\Vert_{m,\infty}^p.\nonumber
\end{align}
Letting $k\rightarrow\infty$, we get that\footnote{We use the convention that $\prod_{j=0}^{-1}=1$.}
\begin{align}
\tilde{\mathbb{E}}\left[\vert\mu_t^n(\varphi)\vert^p\right]=&\Vert\mu_t^n(\varphi)\Vert_p^p\nonumber\\
\leq&2^{p-1}\Delta\Vert\varphi\Vert_{m,\infty}^p+2^{2(p-1)}\beta^pK C^pt^\kappa\Delta\Vert\varphi\Vert_{m,\infty}^p+2^{3(p-1)}\beta^{2p}K^2C^{2p}\frac{t^{2\kappa}}{\kappa+1}\Delta\Vert\varphi\Vert_{m,\infty}^p\nonumber\\
&+\cdots+\frac{2^{k(p-1)}\beta^{(k-1)p}K^{k-1}C^{(k-1)p}t^{(k-1)\kappa}}{(\kappa+1)(2\kappa+1)(3\kappa+1)\cdots((k-2)\kappa+1)}\Delta\Vert\varphi\Vert_{m,\infty}^p\nonumber\\
&+\cdots\cdots\nonumber\\
=&2^{p-1}(\alpha+1)^p\frac{\gamma}{n^q}\Vert\varphi\Vert_{m,\infty}^p\sum_{k=1}^\infty\left[2^{(k-1)(p-1)}\beta^{(k-1)p}K^{k-1}C^{(k-1)p}\frac{t^{(k-1)\kappa}}{\prod_{j=0}^{k-2}(j\kappa+1)}\right]\nonumber.
\end{align}
Let $\eta_{t,k}=2^{(k-1)(p-1)}\beta^{(k-1)p}K^{k-1}C^{(k-1)p}\frac{t^{(k-1)\kappa}}{\prod_{j=0}^{k-2}(j\kappa+1)}$, we know $\xi_t\triangleq\sum_{k=1}^\infty\eta_{t,k}$ exists by the following ratio test 
$$
\lim_{k\rightarrow\infty}\frac{\eta_{t,k+1}}{\eta_{t,k}}=2^{p-1}\beta^pKC^p\frac{t^\kappa}{(k-1)\kappa+1}=0<1.
$$
Finally the result \eqref{eq.main_general_result} follows by setting $c_t=2^{p-1}(\alpha+1)^p\gamma\xi_t$.



\end{document}